\documentclass[11pt]{amsart}
\usepackage{amsfonts,amssymb,amsmath,amsthm}
\usepackage[english]{babel}
\usepackage{mathrsfs}
\usepackage{geometry}
\usepackage{enumerate}
\usepackage{graphicx,color}

\theoremstyle{plain}

\theoremstyle{plain}
\newtheorem{theorem}{Theorem} [section]
\newtheorem{corollary}[theorem]{Corollary}
\newtheorem{lemma}[theorem]{Lemma}
\newtheorem{proposition}[theorem]{Proposition}

\theoremstyle{definition}

\newtheorem{remark}[theorem]{Remark}
\numberwithin{theorem}{section}
\numberwithin{equation}{section}
\numberwithin{figure}{section}

\def\mean#1{\mathchoice
         {\mathop{\kern 0.2em\vrule width 0.6em height 0.69678ex depth -0.58065ex
                 \kern -0.8em \intop}\nolimits_{\kern -0.4em#1}}%
         {\mathop{\kern 0.1em\vrule width 0.5em height 0.69678ex depth -0.60387ex
                 \kern -0.6em \intop}\nolimits_{#1}}%
         {\mathop{\kern 0.1em\vrule width 0.5em height 0.69678ex
             depth -0.60387ex
                 \kern -0.6em \intop}\nolimits_{#1}}%
         {\mathop{\kern 0.1em\vrule width 0.5em height 0.69678ex depth -0.60387ex
                 \kern -0.6em \intop}\nolimits_{#1}}}

\newcommand{\footnoteremember}[2]{
\footnote{#2}
\newcounter{#1}
\setcounter{#1}{\value{footnote}}
}

\def\N{\mathbb N}

\def\R{\mathbb R}

\def\e{\varepsilon}

\def\l{\lambda}

\def\n{\nabla}

\newcommand\zero{{\mathbf 0}}

\newcommand{\p}{\partial}

\def\om{\omega}

\def\U{\mathcal U}

\def\C{{\bf{\mathcal C}}}

\def\1{\mathbf 1}

\def\pac{\partial_{c}}
\def\pacs{\partial_{c^{*}}}

\DeclareMathOperator*{\osc}{osc}
\DeclareMathOperator*{\argmin}{argmin}
\DeclareMathOperator{\dist}{dist}
\DeclareMathOperator{\diam}{diam}

\DeclareMathOperator{\spt}{spt}
\DeclareMathOperator{\co}{co}
\DeclareMathOperator{\Id}{Id}

\DeclareMathOperator{\cexp}{c-exp}
\DeclareMathOperator{\crhoexp}{c_{\rho}-exp}
\DeclareMathOperator{\cstarexp}{c*-exp}

\title{Partial regularity for optimal transport maps}

\author[G. De Philippis]{Guido De Philippis}
\address{Scuola Normale Superiore,
p.za dei Cavalieri 7, I-56126 Pisa, Italy}
\email{guido.dephilippis@sns.it}

\author[A. Figalli]{Alessio Figalli}
\address{Department of Mathematics,
The University of Texas at Austin, 1 University Station C1200,
Austin TX 78712, USA}
\email{figalli@math.utexas.edu}


\begin{document}
\begin{abstract}
We prove that, for general cost functions on $\R^n$, or for the cost $d^2/2$ on a Riemannian manifold,
optimal transport maps between smooth densities are always smooth outside a closed singular set of measure zero.
\end{abstract}

\maketitle

\section{Introduction}

A natural and important issue in optimal transport theory
is the 
regularity of optimal transport maps.
Indeed, apart from being a typical PDE/analysis question, knowing whether optimal maps are smooth or not
is an important step towards a qualitative understanding
of them.

It is by now well known that, for the smoothness of optimal maps, conditions on both the cost function
and on the geometry of the supports of the measures are needed.

In the special case $c(x,y)=|x-y|^2/2$ on $\mathbb R^n$,
Caffarelli \cite{Caf1,Caf2,Caf3,Caf4} proved regularity of optimal maps under suitable assumptions on
the densities and on the geometry of their
support.
More precisely, in its simplest form, Caffarelli's result states as follows:

\begin{theorem}
Let $f$ and $g$ be smooth probability densities, respectively bounded away from zero and infinity on two bounded open sets $X$ and $Y$,
and let $T:X\to Y$ denote the unique optimal transport map from $f$ to $g$ for the quadratic cost $|x-y|^2/2$.
If $Y$ is convex, then $T$ is smooth inside $X$.
On the other hand, if $Y$ is not convex, then there exist smooth densities $f$ and $g$ (both bounded away from zero and infinity on $X$ and $Y$, respectively)
for which the map $T$ is not continuous.
\end{theorem}

A natural question which arises from the previous result is whether one may prove some partial regularity on $T$ when the convexity assumption on $Y$ is removed.
In \cite{F-partial,FK-partial} the authors proved the following result:

\begin{theorem}
Let $f$ and $g$ be smooth probability densities, respectively bounded away from zero and infinity  on two bounded open sets $X$ and $Y$,
and let $T:X\to Y$ denote the unique optimal transport map from $f$ to $g$ for the quadratic cost $|x-y|^2/2$.
Then there exist two open sets
$X'\subset X$ and $Y'\subset Y$, with $|X\setminus X'|=|Y\setminus Y'|=0$,
such that $T:X' \to Y'$ is a smooth diffeomorphism.
\end{theorem}

In the case of general cost functions on $\R^n$, or when $c(x,y)=d(x,y)^2/2$ on a Riemannian manifold $M$ ($d(x,y)$ being the Riemannian distance), the situation is much more complicated.
Indeed, as shown by Ma, Trudinger,
and Wang \cite{MTW}, and Loeper \cite{Loe1},
in addition to suitable convexity assumptions on the support of the target density (or on the cut locus of the manifold when ${\rm supp}(g)=M$ \cite{FRV-nec}),
a very strong structural condition on the cost function, the so-called \textit{MTW condition}, is needed to ensure the smoothness of the map.

More precisely, if the MTW condition holds (together with some suitable convexity assumptions on the target domain), then the optimal map is smooth \cite{TW1,TW2,FL,LTWC2a,FKM-C1a}.
On the other hand, if the MTW condition fails at one point, then one can construct smooth densities (both supported on domains
which satisfy the needed convexity assumptions) for which the optimal transport map is not continuous \cite{Loe1} (see also \cite{F-bourbaki}).

In the case of Riemannian manifolds, the MTW condition for $c=d^2/2$ is very restrictive: indeed, as shown by Loeper \cite{Loe1}, it implies that $M$
has non-negative sectional curvature, and actually it is much stronger than the latter \cite{Kim,FRV-surfaces}.
In particular, if $M$ has negative sectional curvature, then the MTW condition fails at \emph{every} point.
Let us also mention that, up to now, the MTW condition is known to be satisfied only for very special classes of Riemannian manifolds,
such as spheres, their products, their quotients and submersions, and their perturbations \cite{Loe2,FR,DG,KMC2,FRVsn,FKM-sphere,DR},
and for instance it is known to fail on sufficiently flat ellipsoids \cite{FRV-surfaces}.\\

The goal of the present paper is to show that, even without any condition on the cost function or on the supports of the densities,
optimal transport maps  are always smooth outside a closed singular set of measure zero.
In order to state our results, we first have to introduce some basic assumptions on the cost functions which are needed to ensure 
existence and uniqueness of optimal maps.
As before, $X$ and $Y$ denote two open subsets of $\R^n$.
\begin{enumerate}
\item[{\bf (C0)}] The cost function $c:X\times Y\to \R$ is of class $C^{2,\alpha}$ for some $\alpha \in (0,1)$,
with $\|c\|_{C^{2,\alpha}(X\times Y)}<\infty$.
\item[{\bf (C1)}] For any $x \in X$, the map $Y \ni y \mapsto -D_x c(x, y) \in \R^n$ is injective.
\item[{\bf (C2)}] For any $y \in Y$, the map $X \ni x \mapsto -D_y c(x, y) \in \R^n$ is injective.
\item[{\bf (C3)}] $\det(D_{xy} c)(x,y) \neq 0$ for all $(x,y) \in X \times Y$.
\end{enumerate}

Here are our main results:
\begin{theorem}
\label{thm:c}
Let $X,Y\subset \R^n$ be two bounded open sets, and let
$f:X\to \R^+$ and $g:Y\to \R^+$ be two
continuous
probability densities, respectively 
bounded away from zero and infinity on $X$ and $Y$.
Assume that the cost $c:X\times Y \to \R$ satisfies {\bf (C0)}-{\bf (C3)}, and denote by $T:X\to Y$ the unique optimal transport map sending $f$ onto $g$.
Then there exist two relatively closed sets $\Sigma_X \subset X,\Sigma_Y\subset Y$ of measure zero such that
$T:X\setminus \Sigma_X \to Y\setminus \Sigma_Y$ is a homeomorphism of class $C_{\rm loc}^{0,\beta}$ for any $\beta<1$.
In addition, if $c\in C^{k+2,\alpha}_{\rm loc}(X\times Y)$, $f\in C^{k,\alpha}_{\rm loc}(X)$, and $g\in C^{k,\alpha}_{\rm loc}(Y)$
for some $k \geq 0$ and $\alpha \in (0,1)$, then $T:X\setminus \Sigma_X\to Y\setminus \Sigma_Y$ is a
diffeomorphism of class $C_{\rm loc}^{k+1,\alpha}$.
\end{theorem}

\begin{theorem}
\label{cor:M}
Let $M$ be a smooth Riemannian manifold, and let $f,g:M\to \R^+$ be two continuous probability densities, locally bounded away from zero and infinity on $M$.
Let $T:M\to M$ denote the optimal transport map for the cost $c=d^2/2$ sending $f$ onto $g$.
Then there exist two closed sets $\Sigma_X,\Sigma_Y\subset M$ of measure zero such that
 $T:M\setminus \Sigma_X \to M\setminus \Sigma_Y$ is a homeomorphism of class $C_{\rm loc}^{0,\beta}$ for any $\beta<1$.
 In addition, if both $f$ and $g$ are of class $C^{k,\alpha}$, then
 $T:M\setminus \Sigma_X \to M\setminus \Sigma_Y$ is a diffeomorphism of class $C_{\rm loc}^{k+1,\alpha}$.
\end{theorem}

The paper is structured as follows:
 in the next section we introduce some notation and preliminary results.
 Then, in Section \ref{sect:loc}, we show how both Theorem \ref{thm:c} and Theorem \ref{cor:M}
 are a direct consequence of some local regularity results around differentiability points of $T$, see Theorems \ref{c1alpha} and \ref{C2alpha}.
Finally, Sections \ref{sect:C1a} and \ref{sect:C2a} are devoted to the proof of these local results.\\

\textit{Acknowledgements:} We wish to thank Luigi Ambrosio for his careful reading of a preliminary
version of this manuscript. AF is partially supported by NSF Grant DMS-0969962.
Both authors acknowledge the support of the ERC ADG Grant GeMeThNES.

\section{Notation and preliminary results}
Through a well established procedure, maps that solve optimal
transport problems derive from a $c$-convex
potential,  itself solution to a Monge-Amp\`ere type equation.

More precisely, given a cost function $c:X\times Y\to \R$, a function $u:X\to \R$ is said \textit{$c$-convex} if it can be written as 
\begin{equation}
\label{eq:cconvex}
u(x)=\sup_{y \in Y} \left\{-c(x,y)+\lambda_y\right\},
\end{equation}
for some constants $\lambda_y \in \R \cup\{-\infty\}$.

Similarly to the subdifferential for convex function, for $c$-convex functions one can talk about their $c$-subdifferential:
if $u:X\to \R$ is a $c$-convex function as above, the {\em $c$-subdifferential} of $u$ at $x$ is the
(nonempty) set
\begin{equation}
\label{eq:csubdiff}
\p_c u(x) := \{y \in \overline Y : \ u(z)\geq -c(z,y)+c(x,y)+u(x) \quad \forall\,z \in X\}.
\end{equation}
If \(x_{0}\in X\) and $y_{0} \in \p_cu(x_{0})$, we will say that the function 
\begin{equation}
\label{eqCxy0}
C_{x_{0},y_{0}}(\cdot):=-c(\cdot,y_{0})+c(x_{0},y_{0})+u(x_{0})
\end{equation}
 is a \textit{$c$-support} for $u$ at $x_{0}$.
 We also define the \textit{Frechet subdifferential }of $u$ at $x$ as
$$
\p^- u(x) := \bigl\{p \in \R^n : \ u(z) \geq u(x)+p \cdot (z-x)+o(|z-x|)\bigr\}.
$$
We will use the following notation: if $E\subset X$ then
$$
\p_cu(E):=\bigcup_{x\in E}\p_cu(x),\qquad \p^-u(E):=\bigcup_{x\in E}\p^-u(x).
$$
It is easy to check that, if $c$ is of class $C^1$, then the following inclusion holds:
\begin{equation}
\label{eq:rel subdiff}
y \in \p_cu(x) \quad \Longrightarrow\quad -D_xc(x,y) \in \p^- u(x).
\end{equation}
 In addition, if $c$ satisfies {\bf (C0)}-{\bf (C2)}, then we can define the \textit{$c$-exponential map}:
\begin{equation}
\label{eq:cexp}
\text{for any $x\in X$, $y \in Y$, $p \in \R^n$},
\qquad\left\{
\begin{array}{l}
\cexp_x(p)=y \quad \Leftrightarrow \quad p=-D_xc(x,y)\\
\cstarexp_y(p)=x \quad \Leftrightarrow \quad p=-D_yc(x,y)
\end{array}
\right.
\end{equation}
Using \eqref{eq:cexp}, we can rewrite \eqref{eq:rel subdiff} as
\begin{equation}
\label{eq:rel subdiff2}
\p_cu(x) \subset \cexp_x\left(\p^- u(x)\right).
\end{equation}
Notice that, if $c \in C^1$ and $Y$ is bounded,
it follows immediately from \eqref{eq:cconvex} that $c$-convex functions are Lipschitz, so
in particular they are differentiable a.e.

The following notation will be convenient: given a $c$-convex function $u:X\to \R$, we define (at almost every point) the map $T_u:X \to Y$ as
\begin{equation}
\label{eq:Tu}
T_u(x):= \cexp_x(\n u(x)).
\end{equation}
(Of course $T_u$ depends also on $c$, but to keep the notation lighter we prefer not to make this dependence  explicit.
The reader should keep in mind that, whenever we write $T_u$, the cost $c$ is always the one for which $u$ is $c$-convex.)

Finally, let us observe that if $c$ satisfies {\bf (C0)} and $Y$ is bounded, then 
it follows from \eqref{eq:cconvex}
that $u$ is semiconvex (i.e., there exists a constant $C>0$ such that $u+C|x|^2/2$ is convex,
see for instance \cite{FF}).
In particular, by Alexandrov's Theorem, $c$-convex functions are twice differentiable a.e.
(see \cite[Theorem 14.25]{V} for a list of different equivalent definitions of this notion).

The following is a basic result in optimal transport theory (see for instance \cite[Chapter 10]{V}):
\begin{theorem}\label{cbrenier}
Let $c:X\times Y \to \R$ satisfy {\bf (C0)}-{\bf (C1)}.
Given two probability densities $f$ and $g$ supported on $X$ and $Y$ respectively,
there exists a $c$-convex function $u:X\to \R$ such that $T_u:X\to Y$ is the unique optimal transport map sending $f$ onto $g$.
\end{theorem}
In the particular case $c(x,y)=-x\cdot y$ (which is equivalent to the quadratic cost $|x-y|^2/2$), $c$-convex functions
are convex and the above result takes the following simple form \cite{brenier}:
\begin{theorem}
\label{thm:Brenier}
Let $c(x,y)=-x\cdot y$.
Given two probability densities $f$ and $g$ supported on $X$ and $Y$ respectively,
there exists a convex function $v:X\to \R$ such that $T_v=\nabla v:X\to Y$  is the unique optimal transport map sending $f$ onto $g$.
\end{theorem}
Although on Riemannian manifolds the cost function $c=d^2/2$ is not smooth everywhere, one can still prove existence of optimal maps \cite{McC,FF,FG}
(let us remark that, in this case, the $c$-exponential map coincides with the classical exponential map in Riemannian geometry):
\begin{theorem}
\label{thm:M}
Let $M$ be a smooth Riemannian manifold, and $c=d^2/2$.
Given two probability densities $f$ and $g$ supported on $M$,
there exists a $c$-convex function $u:M\to \R\cup\{+\infty\}$ such that
$u$ is differentiable $f$-a.e., and $T_u(x)=\exp_x(\n u(x))$  is the unique optimal transport map sending $f$ onto $g$.
\end{theorem}

We conclude this section by recalling that $c$-convex functions arising in optimal transport problems solve a Monge-Amp\`ere type equation almost everywhere,
referring to  \cite[Section 6.2]{AGS}, \cite[Chapters 11 and 12]{V}, and \cite{F-bourbaki} for more details.

Whenever $c$ satisfies {\bf (C0)}-{\bf (C3)}, then the transport condition $(T_u)_\sharp f=g$ gives
\begin{equation}
\label{eq:detT}
|\det(DT_u(x))|=\frac{f(x)}{g(T_u(x))}\qquad \text{a.e.}
\end{equation}
In addition, the $c$-convexity of $u$ implies that, at every point $x$ where $u$ is twice differentiable,
\begin{equation}
\label{eq:pos def}
D^2u(x)+D_{xx}c\bigl(x,\cexp_x(\nabla u(x))\bigr) \geq 0.
\end{equation}
Hence, writing \eqref{eq:Tu} as
$$
-D_xc(x,T_u(x))=\nabla u(x),
$$
differentiating the above relation with respect to $x$,
and using \eqref{eq:detT} and \eqref{eq:pos def}, we obtain 
\begin{equation}
\label{eq:MA}
\det\Bigl(D^2u(x)+D_{xx}c\bigl(x,\cexp_x(\nabla u(x))\bigr) \Bigr)=\left|\det\left(D_{xy}c\bigl(x,\cexp_x(\nabla u(x))\bigr) \right) \right| \frac{f(x)}{g(\cexp_x(\nabla u(x)))}
\end{equation}
at every point $x$ where $u$ it is twice differentiable.
In particular, when $c(x,y)=-x\cdot y$, the convex function $v$ provided by Theorem \ref{thm:Brenier}
solves the classical Monge-Amp\`ere equation
$$
\det\bigl(D^2v(x) \bigr)=\frac{f(x)}{g(\nabla v(x))}\qquad \text{a.e.}
$$

\section{The localization argument and proof of the results}\label{sect:loc}

The goal of this section is to prove Theorems \ref{thm:c}
and \ref{cor:M} by showing that the 
assumptions of Theorems \ref{c1alpha} and \ref{C2alpha} below are satisfied near almost every point.

The rough idea is the following: if $\bar x$ is a point where the semiconvex function $u$
is twice differentiable, then around that point $u$ looks like a parabola. In addition, by looking close enough to $\bar x$,
the cost function $c$ will be very close to the linear one and the densities will be almost constant there.
Hence we can apply Theorem \ref{c1alpha} to
deduce that $u$ is of class $C^{1,\beta}$ in neighborhood of $\bar x$
(resp. $u$ is of class $C^{k+2,\alpha}$ by Theorem \ref{C2alpha}, if $c \in C_{\rm loc}^{k+2,\alpha}$ and $f,g\in C_{\rm loc}^{k,\alpha}$),
which implies in particular that $T_u$  is of class $C^{0,\beta}$ in neighborhood of $\bar x$
(resp. $T_u$ is of class $C^{k+1,\alpha}$ by Theorem \ref{C2alpha}, if $c \in C_{\rm loc}^{k+2,\alpha}$ and $f,g\in C_{\rm loc}^{k,\alpha}$).
Being our assumptions completely symmetric in $x$ and $y$, we can apply the same argument to the optimal map $T^*$ sending $g$ onto $f$.
Since $T^*=(T_u)^{-1}$ (see the discussion below), it follows that $T_u$ is a global
homeomorphism of class $C_{\rm loc}^{0,\beta}$  (resp. $T_u$ is a global diffeomorphism of class $C_{\rm loc}^{k+1,\alpha}$) outside a closed set of measure zero.

We now give a detailed proof.

\begin{proof}[Proof of Theorem \ref{thm:c}]
Let us introduce the ``$c$-conjugate'' of $u$, that is, the function $u^c:Y\to \R$ defined as 
\[
u^{c}(y):=\sup_{x\in X}\big\{-c(x,y)-u(x)\big\}.
\]
Then \(u^{c}\) is \(c^{*}\)-convex, where
\begin{equation}
\label{eq:dual}
 c^{*}(y,x):=c(x,y), \qquad \text{and} \qquad x\in \partial_{c^{*}}u^{c}(y)\quad \Leftrightarrow \quad y\in \pac u(x)
\end{equation}
(see for instance \cite[Chapter 5]{V}). 

Being our assumptions completely symmetric in \(x\) and \(y\), \(c^{*}\) satisfies the same assumptions as $c$.
In particular, by Theorem \ref{cbrenier}, there exists an optimal  map \(T^{*}\)  (with respect to \(c^{*}\)) sending \(g\) onto \(f\). In addition, it is well-known that $T^*$ is actually equal to
\[
T_{u^{c}}(y)=\cstarexp_{y}\big(\nabla u^{c} (y)\big),
\]
and that $T_u$ and $T_{u^c}$ are inverse to each other, that is
\begin{equation}\label{inverse}
 T_{u^{c}}\big(T_{u}(x)\big)=x, \quad T_{u}\big (T_{u^{c}}(y)\big)=y \qquad \text{for a.e. \(x\in X\), \(y\in Y\)}
\end{equation}
(see, for instance, \cite[Remark 6.2.11]{AGS}).

Since semiconvex functions are twice differentiable a.e., there exist sets $X_{1}\subset X,Y_{1}\subset Y$ of full measure 
such that \eqref{inverse} holds for every $x \in X_1$ and $y \in Y_1$,
and in addition $u$ is twice differentiable for every $x \in X_1$ and \(u^{c}\) is twice differentiable for every \(y\in Y_{1}\).  Let us define 
\[
X':=X_{1}\cap (T_{u})^{-1}(Y_{1}).
\]
Using  that \(T_{u}\) transports  \(f\) on \(g\) and that the two densities are bounded away from zero and infinity, we see that \(X'\) is of full measure in \(X\).

We fix a point $\bar x \in X'$. Since $u$ is differentiable at $\bar x$ (being twice differentiable), it follows by \eqref{eq:rel subdiff2} that the set $\p_cu(\bar x)$ is a singleton, namely $\p_cu(\bar x)=\{\cexp_{\bar x}(\nabla u(\bar x))\}$.
Set $\bar y:=\cexp_{\bar x}(\nabla u(\bar x))$. Since $\bar y \in Y_1$ (by definition of $X'$),
\(u^{c}\) is twice differentiable at \(\bar y\) and \(\bar x=T_{u^{c}}(\bar y)\). Up to a translation in the system of coordinates (both in $x$ and $y$)
we can assume that both $\bar x$ and $\bar y$ coincide with the origin $\zero$. 

Let us define
\begin{equation*}
\begin{split}
&\bar u(z):=u(z)-u(\zero)+c(z,\zero)-c(\zero,\zero),\\
&\bar c(z,w):=c(z,w)-c(z,\zero)-c(\zero,w)+c(\zero,\zero),\\
&\bar u^{\bar c}(w):=u^{c}(w)-u(\zero)+c(\zero,w)-c(\zero,\zero).
\end{split}
\end{equation*}
Then $\bar u$ is a $\bar c$-convex function, \(\bar u^{\bar c}\) is its \(\bar c\)-conjugate,  $T_{\bar u}=T_u$, and \(T_{\bar u^{\bar c}}=T_{u^{c}}\), so
in particular $\left(T_{\bar u}\right)_\sharp f=g$ and \(\left(T_{\bar u^{\bar c}}\right)_\sharp g=f\). 
In addition, because by assumption $\zero \in X'$, $\bar u$ is twice differentiable at $\zero$ and \(\bar u^{\bar c}\) is twice differentiable at \(\zero=T_{\bar u}(\zero)\).
Let us define $P:=D^2\bar u(\zero)$, 
and $M:=D_{xy}\bar c(\zero,\zero)$. Then, since $\bar c(\cdot ,\zero)=\bar c(\zero,\cdot)\equiv 0$
and $\bar c \in C^{2,\alpha}$, a Taylor expansion gives
$$
\bar u(z)=\frac{1}{2} Pz\cdot z+o(|z|^2),\qquad \bar c(z,w)=Mz\cdot w+O(|z|^{2+\alpha}+|w|^{2+\alpha}),
$$
Let us observe that, 
since by assumption $f$ and $g$ are bounded away from zero and infinity,
 by {\bf (C3)} and \eqref{eq:MA} applied to $\bar u$ and $\bar c$
we get that $\det(P),\det(M)\neq 0$.
In addition \eqref{eq:pos def} implies that $P$ is a positive definite symmetric matrix.
Hence, we can perform a second change of coordinates: $z \mapsto \tilde z:=P^{1/2}z$, $w \mapsto \tilde w:=-P^{-1/2}M^*w$ ($M^*$ being the transpose of $M$),
so that, in the new variables,
\begin{equation}
\label{eq:change coord}
\tilde u(\tilde z):=\bar u(z)=\frac{1}{2}|\tilde z|^2 +o(|\tilde z|^2),\qquad  \tilde c(\tilde z,\tilde w):=\bar c(z,w)=-\tilde z\cdot \tilde w+O(|z|^{2+\alpha}+|w|^{2+\alpha}).
\end{equation}

By an easy computation it follows that $(T_{\tilde u})_\sharp \tilde f=\tilde g$, where\footnoteremember{footnote:det}{
An easy way to check this is to observe that the measures $\mu:=f(x)dx$ and $\nu:=g(y)dy$ are independent of the choice of coordinates,
hence \eqref{cambiodens}  follows from the identities
$$
f(x)dx=\tilde f(\tilde x)d\tilde x ,\qquad g(y)dy=\tilde g(\tilde y)d\tilde y.
$$}
\begin{equation}\label{cambiodens}
\tilde f(\tilde z):=\det(P^{-1/2})\,f( P^{-1/2}\tilde z),\qquad \tilde g (\tilde w):=\big|\det\big((M^*)^{-1}P^{1/2}\big)\big|\, g( -(M^*)^{-1}P^{1/2}\tilde w).
\end{equation}
Notice that
\begin{equation}\label{eq:good grad}
D_{\tilde z\tilde z}\tilde c(\zero,\zero)=D_{\tilde w\tilde w}\tilde c(\zero,\zero)=\zero_{n\times n},\qquad  -D_{\tilde z\tilde w}\tilde c(\zero,\zero)=\Id, \qquad D^2\tilde u(\zero)=\Id,
\end{equation}
so, using  \eqref{eq:MA}, we deduce that
\begin{equation}\label{uno}
\frac {\tilde f(\zero)}{\tilde g(\zero)}=\frac{\det\big( D^{2}\tilde u(\zero)+D_{\tilde z\tilde z}\tilde c(\zero,\zero)\big)}{\bigl|\det \bigl(D_{\tilde z\tilde w}\tilde c(\zero,\zero)\bigr)\bigr|}=1.
\end{equation}

To  ensure that we can apply Theorems \ref{c1alpha} and \ref{C2alpha}, we now perform the following dilation: for $\rho>0$ we define
$$
u_\rho(\tilde z):= \frac{1}{\rho^2} \tilde u(\rho\tilde z),\qquad  c_\rho(\tilde z,\tilde w):=\frac{1}{\rho^2}\tilde c(\rho\tilde z,\rho\tilde w).
$$
We claim that, provided $\rho$ is sufficiently small, $u_\rho$ and $c_\rho$ satisfy the assumptions of Theorems \ref{c1alpha} and \ref{C2alpha}.

Indeed, it is immediate to check that $u_\rho$ is a $c_\rho$-convex function.
Also, by the same argument as above,
from the relation $(T_{\tilde u})_\sharp \tilde f=\tilde g$ we deduce that $T_{u_\rho}$ sends
$\tilde f(\rho\tilde z)$ onto $\tilde g(\rho \tilde w)$.
In addition, since we can freely multiply both densities by a same constant, it actually follows from 
\eqref{uno} that $(T_{u_\rho})_\sharp f_\rho=g_\rho$, where
$$
f_{\rho}(\tilde z):=\frac{\tilde f(\rho\tilde z)}{\tilde f(\zero)},\qquad g_\rho (\tilde w):=\frac{\tilde g(\rho \tilde w)}{\tilde g(\zero)}.
$$
In particular, since $f$ and $g$ are continuous, we get
\begin{equation}
\label{eq:fgrho}
|f_\rho -1|+ |g_{\rho}-1| \to 0\qquad \text{inside $B_{3}$}
\end{equation}
as $\rho \to 0$. 
Also, by \eqref{eq:change coord} we get that, for any $\tilde z,\tilde w \in B_3$,
\begin{equation}
\label{eq:almost linear}
u_{\rho}(\tilde z)=\frac{1}{2}|\tilde z|^2 +o(1),\qquad  c_\rho(\tilde z,\tilde w)=-\tilde z\cdot \tilde w+O(\rho^\alpha|z|^{2+\alpha}+\rho^\alpha|w|^{2+\alpha}).,
\end{equation}
where $o(1)\to 0$ as $\rho \to 0$.
In particular, it follows easily that \eqref{eq:costconv2} and \eqref{eq:parabola} hold with any positive constants
$\delta_0,\eta_0$ provided $\rho$ is small enough.


Furthermore, by the second order differentiability of $\tilde u$ at $\zero$ it follows that
the multivalued map $\tilde z \mapsto \p^-\bar u(\tilde z)$ is differentiable at $\zero$ (see \cite[Theorem 14.25]{V})
with gradient equal to the identity matrix (see \eqref{eq:change coord}),
hence
$$
\partial^-u_\rho(\tilde z)\subset B_{\gamma_\rho}(\tilde z) \qquad \forall\, \tilde z \in B_2,
$$
where $\gamma_\rho \to 0$ as $\rho \to 0$.
Since $\partial_{c_\rho}u_\rho \subset \crhoexp\left( \partial^-u_\rho\right)$ (by \eqref{eq:rel subdiff2}) and $\|\crhoexp -\Id\|_{\infty}=o(1)$ (by \eqref{eq:almost linear}),
we get
\begin{equation}\label{cdiff}
\partial_{c_\rho}u_\rho(\tilde z)\subset B_{\delta_\rho}(\tilde z) \qquad \forall\, \tilde z \in B_3,
\end{equation}
with \(\delta_{\rho}=o(1)\) as \(\rho \to 0\). Moreover, the \(c_{\rho}\)-conjugate of \(u_{\rho}\) is easily seen to be
\[
u_{\rho}^{c_{\rho}}(\tilde w)=\bar u ^{\bar c }\big (\rho(M^*)^{-1}P^{1/2}\tilde w\big).
\]
Since \(u^{c}\) is twice differentiable at \(\zero\), so is \( u_{\rho}^{c_{\rho}}\). In addition,
an easy computation \footnote{For instance, this follows by differentiating both relations
$$
D_{\tilde z}c_\rho\big(\tilde z,T_{u_\rho}(\tilde z)\big)=-\nabla u_\rho(\tilde z)\quad\text{and}
\quad D_{\tilde w}c_\rho\big(T_{u_\rho^{c_\rho}}(\tilde w),\tilde w\big)=-\nabla u_\rho^{c_\rho}(\tilde w)
$$
at \(\zero\), and 
using then \eqref{eq:good grad}
and the fact that
\(\nabla T_{u_{\rho}^{c_{\rho}}}(\zero)=[\nabla T_{u_{\rho}}(\zero)]^{-1}$ and $D^2u_\rho(\zero)=\Id$.}
shows that \(D^{2}u_{\rho}^{c_{\rho}}(\zero)=\Id\). Hence, arguing as above we obtain that 
\begin{equation}\label{cstardiff}
\partial_{c^{*}_\rho}u^{c_\rho}_\rho(\tilde w)\subset B_{\delta_\rho'}(\tilde w) \qquad \forall\, \tilde w \in B_3,
\end{equation}
with \(\delta'_{\rho}=o(1)\) as \(\rho \to 0\).

We now define 
\[
\mathcal C_{1}:=\overline{B}_1,\qquad
\mathcal C_{2}:=\partial_{c_{\rho}}u_{\rho}(\mathcal C_1).
\]
Observe that both $\mathcal C_1$ and $\mathcal C_2$ are closed
(since the $c$-subdifferential of a compact set is closed). Also, thanks to \eqref{cdiff}, by choosing \(\rho\) small enough we can ensure that
\(B_{1/3}\subset \mathcal C_{2}\subset B_{3}\). Finally, it follows from \eqref{eq:rel subdiff2}
that
\[
(T_{u_{\rho}})^{-1}(\mathcal C_{2})\setminus \mathcal C_{1} \subset
(T_{u_{\rho}})^{-1}\bigl(\{\text{points of non-differentiability of \(u_{\rho}^{c_\rho}\)}\}\bigr),
\]
and since this latter set has measure zero, a simple computation shows that
\[
\big(T_{u_{\rho}}\big)_{\sharp} (f_\rho\mathbf 1_{\mathcal C_{1}})=g_\rho\mathbf 1_{\mathcal C_{2}}.
\]

Thus, thanks to  \eqref{eq:fg}, we get that for any $\beta<1$ the assumptions of Theorem \ref{c1alpha} are satisfied, provided we choose $\rho$ sufficiently small.
Moreover, if in addition $c \in C_{\rm loc}^{k+2,\alpha}(X\times Y)$,
$f\in C_{\rm loc}^{k,\alpha}(X)$,
and $g\in C_{\rm loc}^{k,\alpha}(Y)$, then also the assumptions of Theorem \ref{C2alpha} are satisfied.

Hence, by applying Theorem \ref{c1alpha} (resp. Theorem \ref{C2alpha}) we deduce that
$u_\rho \in C^{1,\beta}(B_{1/7})$
(resp. $u_\rho \in C^{k+2,\alpha}(B_{1/9})$), so going
back to the original variables we get the existence of a neighborhood $\U_{\bar x}$
of $\bar x$ such that $u \in C^{1,\beta}(\U_{\bar x})$ (resp. $u \in C^{k+2,\alpha}(\U_{\bar x})$).
This implies in particular that $T_u \in C^{0,\beta}(\U_{\bar x})$ (resp. $T_u \in C^{k+1,\alpha}(\U_{\bar x})$). Moreover, it follows by Corollary \ref{open} that
\(T_{u}(\U_{\bar x})\) contains a neighborhood of $\bar y$.

We now observe that, by symmetry, we can also apply Theorem \ref{c1alpha} (resp. Theorem \ref{C2alpha}) to $u_\rho^{c_\rho}$.
Hence, there exists a neighborhood \(\mathcal V_{\bar y}\) of $\bar y$ such that \(T_{u^{c}}\in C^{0,\beta}(\mathcal V_{\bar y})\).
Since $T_u$ and $T_{u^c}$ are inverse to each other (see \eqref{inverse}) we deduce that,
possibly reducing the size of \(\U_{\bar x}\),
$T_u$ is a homeomorphism (resp. diffeomorphism) between $\U_{\bar x}$ and $T_u(\U_{\bar x})$.
Let us consider the open sets
$$
X'':=\bigcup_{\bar x \in X'} \U_{\bar x},\qquad Y'':=\bigcup_{\bar x \in X'} T_u(\U_{\bar x}),
$$
and define the (relatively) closed $\Sigma_X:=X\setminus X''$,  $\Sigma_Y:=Y\setminus Y''$. 
Since $X''\supset X'$, $X''$ is a set of full measure, so $|\Sigma_X|=0$.
In addition, since $\Sigma_Y=Y\setminus Y''\subset Y\setminus T_u( X')$ and $T_u(X')$ has full measure in \(Y\),
we also get that $|\Sigma_Y|=0$.

Finally, since $T_u:X\setminus \Sigma_X\to Y\setminus \Sigma_Y$ is a local homeomorphism (resp. diffeomorphism),
by \eqref{inverse} it follows that $T_u:X\setminus \Sigma_X\to Y\setminus \Sigma_Y$ is a global homeomorphism (resp. diffeomorphism), which concludes the proof.
\end{proof}

\begin{proof}[Proof of Theorem \ref{cor:M}]
The only difference with respect to the situation in Theorem \ref{thm:c} is that now the cost function $c=d^2/2$ is not smooth on the whole $M\times M$.
However, even if $d^2/2$ is not everywhere smooth and $M$ is not necessarily compact, it is still true that the $c$-convex function $u$ provided by Theorem \ref{thm:M} is locally semiconvex (i.e., it is locally semiconvex when seen in any chart) \cite{FF,FG}.
In addition, as shown in \cite[Proposition 4.1]{CMS}
(see also \cite[Section 3]{F}), if $u$ is twice differentiable at $x$, then the point
$T_u(x)$ is not in the cut-locus of $x$.
Since the cut-locus is closed and $d^2/2$ is smooth outside the cut-locus, we deduce the
existence of a set $X$ of full measure such that, if $x_0 \in X$, then: (1) $u$ is twice differentiable at $ x_0$; (2)
there exists a neighborhood $\U_{x_0}\times \mathcal V_{T_u( x_0)}\subset M\times M$ of $( x_0,T_u( x_0))$ such that $c \in C^\infty(\U_{x_0}\times \mathcal V_{T_u( x_0)})$.
Hence, by taking a local chart around $(x_0,T_u(x_0))$, 
the same proof as the one of Theorem \ref{thm:c} shows that $T_u$ is a local homeomorphism (resp. diffeomorphism) around almost every point. 
Using as before that $T_u:M \to M$ is invertible a.e., it follows that $T_u$ is a global homeomorphism (resp. diffeomorphism) outside a closed singular set 
of measure zero.
We leave the details to the interested reader.
\end{proof}

\section{\(C^{1,\beta}\) regularity and strict \(c\)-convexity}\label{sect:C1a}

In this and the next section we prove that, if in some open set a $c$-convex function $u$ is sufficiently close to a parabola and the cost function is close to the linear one,
then $u$ is smooth in some smaller set.

The idea of the proof (which is reminiscent of the argument introduced by Caffarelli in \cite{Caf4} to show $W^{2,p}$ and $C^{2,\alpha}$ estimates for the classical Monge-Amp\`ere equation,
though several additional complications arise in our case) is the following: since the cost function is close to the linear one
and both densities are almost constant, $u$ is close to a convex function $v$ solving an optimal transport problem with linear cost and constant densities
(Lemma \ref{lem:compact}). In addition, since $u$ is close to a parabola, so is $v$. Hence, by 
\cite{FK-partial} and Caffarelli's regularity theory, $v$ is smooth,
and we can use this information to deduce that $u$ is even closer to a second parabola (given by the second order Taylor expansion of $v$ at the origin) inside a small neighborhood around of origin.
By rescaling back this neighborhood at scale $1$ and iterating this construction, we obtain that $u$ is $C^{1,\beta}$ at the origin for some $\beta \in (0,1)$.
Since this argument can be applied at every point in a neighborhood of the origin, we deduce that $u$ is $C^{1,\beta}$ there, see Theorem \ref{c1alpha}.
(A similar strategy has also been used in \cite{CMN} to show regularity optimal transport maps for the cost $|x-y|^p$, either when $p$ is close to $2$
or when $X$ and $Y$ are sufficiently far from each other.)

Once this result is proved, we know that $\partial^-u$ is a singleton at every point, so it follows from \eqref{eq:rel subdiff2} that
$$
\partial_cu(x) = \cexp_x (\partial^-u(x)),
$$
see Remark \ref{localtoglobal} below.
(The above identity is exactly what in general may fail for general $c$-convex functions, unless the MTW condition holds \cite{Loe1}.)
Thanks to this fact, we obtain that $u$ enjoys a comparison principle (Proposition \ref{comparison}),
and this allows us to use a second approximation argument with solutions of the classical Monge-Amp\`ere equation
(in the spirit of \cite{Caf4,JW}) to conclude that $u$ is $C^{2,\sigma'}$ in a smaller neighborhood, for some $\sigma'>0$.
Then higher regularity follows from standard elliptic estimates, see Theorem \ref{C2alpha}.\\

\begin{lemma}\label{lem:compact}Let \(\C_{1}\) and \(\C_{2}\) be two closed sets such that
\begin{equation}\label{eq:frapalle}
B_{1/K}\subset \C_{1}, \C_{2}\subset B_{K}
\end{equation}
for some $K\geq 1$, $f$ and $g$ two densities supported respectively in $\C_1$ and $\C_2$,
and \(u: \C_{1} \to \R\) a \(c\)-convex function such that $\p_cu(\C_1)\subset B_K$ and \((T_{u})_\sharp f =g \).
Let $\rho>0$ be such that $|\C_1|=|\rho\, \C_2|$ (where $\rho\, \C_2$ denotes the dilation of $\C_2$ with respect to the origin),
and let \(v\) be a convex function such that \(\nabla v_{\sharp }\1_{\C_{1}}
=\1_{\rho \C_{2}}\) and \(v(\zero)=u(\zero)\). Then there exists an increasing function $\omega:\R^+\to \R^+$, depending only \(K\), and satisfying
\(\omega(\delta)\ge \delta \) and \(\omega(0^{+})=0\), such that, if
\begin{equation}\label{eq:ll}
\|f- \1_{\C_{1}} \|_{\infty}+ \| g- \1_{\C_{2}}\|_{\infty} \leq {\delta}
\end{equation} 
and
 \begin{equation}\label{eq:costconv}
 \|c(x,y)+x\cdot y\|_{C^{2}(B_{K}\times B_{K})}\le \delta,
 \end{equation}
 then
\[
 \|u-v\|_{C^{0}(B_{1/K})}\le \omega(\delta).
 \] 
  \end{lemma}
\begin{proof}
Assume the lemma is false. Then there exists \(\e_{0}>0\),  a sequence  of closed sets \(\C^{h}_{1}\), \(\C^{h}_{2}\) satisfying \eqref{eq:frapalle},
functions \(f_{h}\), \(g_{h}\) satisfying \eqref{eq:ll} with $\delta=1/h$, and costs \(c_{h}\) converging in \(C^{2}\) to \(-x\cdot y\), such that 
$$
u_{h}(\zero)=v_{h}(\zero)=0 \quad \text{and}\quad \|u_{h}-v_{h}\|_{C^{0}(B_{1/K})} \ge \e_{0},
$$
where \(u_{h}\) and \(v_{h}\) are as in the statement. First, we extend \(u_{h}\) an \(v_{h}\) to \(B_{K}\) as
\[
u_{h}(x):=\sup_{ z\in \mathcal C_1^h,\,y\in \p_{c_h}u_h(z)} \big\{u_h(z)-c_{h}(x,y)+c_h(z,y)\big\},\qquad 
v_{h}(x):=\sup_{ z\in \mathcal C_1^h,\,p\in \p^-v_h(z)} \big\{v_h(z)+p \cdot(x-z)\big\}. 
\]
Notice that, since by assumption
\(\partial_{c_{h}} u_{h}(\mathcal C_1^h)\subset B_K\),
we have \(\partial_{c_{h}} u_{h}(B_{K})\subset B_K\).
Also, \((T_{u_h})_\sharp f_h =g_h \) gives that $\int f_h=\int g_h$, so it follows from \eqref{eq:ll} that
$$
\rho_h=\left(|\C_1^h|/|\C_2^h|\right)^{1/n} \to 1 \qquad \text{as $h \to \infty$,}
$$
which implies that $\p^-v_h(B_K) \subset B_{\rho_h K}\subset B_{2K}$ for $h$ large.
Thus, since the \(C^{1}\)-norm of \(c_{h}\) is uniformly bounded,
we deduce that both
\(u_{h}\) and \(v_{h}\) are uniformly Lipschitz.
Recalling that \(u_{h}(\zero)=v_{h}(\zero)=0\), we get that, up to a subsequence,
\(u_{h}\) and \(v_{h}\) uniformly converge inside $B_K$ to \(u_{\infty}\) and \(v_{\infty}\) respectively, where 
\begin{equation}\label{eq:lontane}
u_{\infty}(\zero)=v_{\infty}(\zero)=0 \quad \text{and}\quad \|u_{\infty}-v_{\infty}\|_{C^{0}(B_{1/K})} \ge \e_{0}.
\end{equation}
In addition \(f_{h}\) (resp. \(g_{h}\)) weak-\(*\) converge in \(L^{\infty}\) to some density \(f_{\infty }\) (resp. \(g_{\infty}\)) supported in \(\overline B_K\).
Also, since
$\rho_h \to 1$, using \eqref{eq:ll} we get that \(\1_{\C_1^h}\) (resp. \(\1_{\rho_h\C_2^h}\)) weak-\(*\) converges in \(L^{\infty}\) to \(f_{\infty }\) (resp. \(g_{\infty}\)).
Finally we remark that, because of \eqref{eq:ll} and the fact that $\C_1^h\supset B_{1/K}$, we also have
$$
f_\infty \geq \1_{B_{1/K}}.
$$

In order to get a contradiction we have to show that \(u_{\infty}=v_{\infty}\) in \(B_{1/K}\). 
To see this, we apply \cite[Theorem 5.20]{V} to deduce that
both $\nabla u_\infty$ and $\nabla v_\infty$ are optimal transport maps for the linear cost $-x\cdot y$
sending $f_\infty$ onto $g_\infty$. By uniqueness of the optimal map (see Theorem \ref{thm:Brenier}) we deduce that \(\nabla v_{\infty}=\nabla u_{\infty}\) almost everywhere inside \(B_{1/K}\subset \spt f_{\infty}\),
hence \(u_{\infty}=v_{\infty}\) in \(B_{1/K}\) (since \(u_{\infty}(\zero)=v_{\infty}(\zero)=0\)), contradicting \eqref{eq:lontane}.
\end{proof}

Here and in the sequel, we use $ \mathcal N_r(E)$ to denote the $r$-neighborhood of a set $E$.
\begin{lemma}\label{sottodif} Let \(u\) and \(v\) be, respectively,  \(c\)-convex and convex, let $D \in \R^{n\times n}$ be a symmetric matrix satisfying
\begin{equation}\label{bound}
\Id/K \le D\le K \Id
\end{equation}
for some \(K\ge 1\), and define the ellipsoid
\[
E(x_{0},h):=\big\{x:\   D(x-x_{0})\cdot (x-x_{0})\le h\big\}, \qquad h>0.
\]
Assume that there exist small positive constants $\e,\delta$ such that
\begin{equation}\label{eq:hp-sottodiff}
\|v-u\|_{C^{0}(E(x_{0},h))}\le \e,\qquad \|c+x\cdot y\|_{C^{2}\left(E(x_{0},h)\times \pac u(E(x_{0},h)\right)}\le \delta.
\end{equation}
Then
\begin{equation}\label{eq:sottodifftes1}
\pac u \big(E(x_{0},h-\sqrt \e )\big)\subset \mathcal N_{K'(\delta +\sqrt{h\e})} \bigl(\partial v (E(x_{0},h))\bigr) \qquad \forall\,0< \e <h^2 \le 1,
\end{equation}
where \(K'\) depends only on \(K\).	
\end{lemma}	

\begin{proof}
Up to a change of coordinates we can assume that \(x_{0}=\zero\), and to simplify notation we set $E_h:=E(x_0,h)$.
Let us define
 \[
 \bar v(x):=v(x)+\e+2\sqrt\e(Dx \cdot x-h),
 \]
 so that
 \(\bar v \ge u\) outside \( E_{h}\), and \(\bar v\le u\) inside \( E_{h-\sqrt \e}\).
 Then, taking a \(c\)-support to \(u\) in \(E_{h-\sqrt \e}\) (i.e., a function $C_{x,y}$ as in \eqref{eqCxy0},
 with $x \in E_{h-\sqrt \e}$ and $y \in \p_cu(x)$), moving it down and then lifting it up until it touches $\bar v$
 from below, we see that it has to touch the graph of \(\bar v\) at some point \(\bar x \in E_{h} \): in other words \footnote{
 Even if $\bar v$ is not $c$-convex, it still makes sense to consider his $c$-subdifferential (notice that the $c$-subdifferential of $\bar v$ may be empty at some points). In particular, the inclusion $\pac \bar v(x)\subset \cexp_x(\p^-\bar v(x))$ still holds.}
\[
\pac u(E_{h-\sqrt \e})\subset \pac \bar v (E_{h}).
\]
By \eqref{bound} we see that $\diam E_{h}\le  2\sqrt{K h}$,
so by a simple computation (using again \eqref{bound}) we get
\[
\partial^{-} \bar v (E_{h})\subset \mathcal N_{4K\sqrt  {Kh\e} }\bigl(\partial^{-}  v (E_{h})\bigr).
\]
Thus, since $\pac \bar v (E_{h}) \subset \cexp\bigl(\partial^- \bar v (E_{h})\bigr)$
(by \eqref{eq:rel subdiff2})
and $\|\cexp-\Id\|_{C^{0}}\le \delta$ (by \eqref{eq:hp-sottodiff}), 
we easily deduce that
\[
\pac u(E_{h-\sqrt\e})\subset \mathcal N_{K'(\delta+\sqrt{h\e})}\bigl(\partial^- v (E_{h})\bigr),
\]
proving \eqref{eq:sottodifftes1}.
\end{proof}

\begin{theorem}\label{c1alpha}Let \(\C_{1}\) and \(\C_{2}\) be two closed sets satisfying
\begin{equation*}
B_{1/3}\subset \C_{1} ,\C_{2}\subset B_{3},
\end{equation*}
let \(f, g\) be 
 two densities supported in $\C_1$ and $\C_2$ respectively, and
let \(u: \C_{1} \to \R\) be a \(c\)-convex function such that $\p_cu(\C_1)\subset B_3$
and \((T_{u})_\sharp f =g\).
Assume also that \(c\in C^{2,\alpha}(B_{3}\times B_{3})\) 
for some $\alpha \in (0,1)$. Then, for every \(\beta\in (0,1)\) there  exist  constants \(\delta _{0},\eta_{0}>0\) such that the following holds: if 
 \begin{equation}\label{eq:fg}
 \|f-  \1_{\C_{1}} \|_{\infty}+ \| g- \1_{\C_{2}}\|_{\infty} \leq \delta_{0},
 \end{equation}
 \begin{equation}\label{eq:costconv2}
 \|c(x,y)+x\cdot y\|_{C^{2,\alpha}(B_{3}\times B_{3})}\le \delta_{0},
 \end{equation}
 and 
\begin{equation}\label{eq:parabola}
\left\|u-\frac{1}{2} |x|^{2}\right\|_{C^{0}(B_{3})}\le \eta_{0},
\end{equation} 
then \(u\in C^{1,\beta}(B_{1/7})\).
\end{theorem}

\begin{proof}We divide the proof into several steps.

\medskip
\noindent 
$\bullet$ \textit{Step 1: \(u\) is close to a strictly convex solution of the Monge Amp\`ere equation.}
Let \(v:\R^n \to \R\) be a convex function such that \(\nabla v_{\sharp} \1_{\C_{1}}=\1_{\rho\C_{2}}\) with $\rho=(|\C_1|/|\C_2|)^{1/n}$ (see Theorem \ref{thm:Brenier}).
Up to adding a constant to $v$, without loss of generality we can assume that $v(\zero)=u(\zero)$.
Hence, we can apply Lemma \ref{lem:compact} to
obtain 
\begin{equation}\label{eq:vicineeps}
\|v-u\|_{C^{0}(B_{1/3})}\le \omega(\delta_0)
\end{equation}
for some (universal) modulus of continuity $\omega:\R^+\to\R^+$, which combined with 
\eqref{eq:parabola} gives
$$
\left\|v-\frac{1}{2} |x|^{2}\right\|_{C^{0}(B_{1/3})}\le \eta_{0}+\omega(\delta_{0}).
$$
Also, since $\int_{\mathcal C_1}  f=\int_{\mathcal C_2} g$,
it follows easily from \eqref{eq:fg} that $|\rho - 1|\leq 3\delta_0$.
By these two facts we get that $\partial^-v(B_{1/4})\subset B_{7/24}\subset  \rho \mathcal C_2$
provided \(\delta_{0}\) and \(\eta_{0}\) are small enough (recall that  $v$ is convex and that
$B_{1/3}\subset \mathcal C_2$), so we can apply
\cite[Proposition 3.4]{FK-partial} to deduce that 
\(v\) is a strictly convex  Alexandrov solution to the Monge-Amp\`ere equation
\begin{equation}
\label{eq:MAv}
\det D^{2}v =1\quad \text{in \(B_{1/4}\).}
\end{equation}
In addition, by a simple compactness argument, we see that the modulus of strict convexity of \(v\) inside $B_{1/4}$ is universal.
So, by classical Pogorelov and Schauder estimates,  we obtain the existence of a universal constant \(K_0\geq 1\) such that
\begin{equation}\label{eq:pog}
\|v\|_{C^{3}(B_{1/5})}\le K_0, \qquad  \Id/K_0 \le D^{2}v\le K_0\Id \quad \text{in }B_{1/5}.
\end{equation}
In particular, there exists a universal value \(\bar h>0\) such that, for all \(x \in B_{1/7}\),
\[
Q(x, v,h):=\big\{z:\  v(z)\le v(x)+\nabla v (x)\cdot (z-x)+h\big\} \subset \subset  B_{1/6} \quad \forall\, h\le \bar h.
\]
\medskip
\noindent
$\bullet$ \textit{Step 2: Sections of \(u\) are close to  sections of \(v\).}
Given $x \in B_{1/7}$ and $y \in \p_cu(x)$, we define
$$
S(x,y,u,h):=\big\{z:\  u(z)\le u(x)-c(z,y)+c(x,y)+h\big\}.
$$
We  claim that, if $\delta_0$ is small enough, then for all \(x\in B_{1/7}\), $y \in \p_cu(x)$, and $h \leq \bar h/2$, it holds
\begin{equation}\label{eq:seccomparison}
Q(x,v,h-K_1\sqrt {\omega(\delta_0)})\subset S(x,y,u,h)\subset Q(x,v,h+K_1\sqrt {\omega(\delta_0)}),
\end{equation}
where \(K_1>0\) is a universal constant.

Let us show the first inclusion. For this, take $x \in B_{1/7}$, \(y\in \pac u (x)\), and define
\[
p_{x}:=-D_xc(x,y)\in  \partial ^{-} u(x).
\]   
Since \(v\) has universal \(C^{2}\) bounds (see \eqref{eq:pog}) and \(u\) is semi-convex (with a universal bound),
a simple interpolation argument gives
\begin{equation}\label{gradientivicini}
|p_{x}-\nabla v(x)|\le K'\sqrt{ \|u-v\|_{C^{0}(B_{1/5})}} \le  K'\sqrt{\om(\delta_0)} \quad \forall \, x\in B_{1/7}.
\end{equation}
In addition, by  \eqref{eq:costconv2},
\begin{equation}
\label{eq:yp}
|y-p_x|\leq  \|D_xc +{\rm Id}\|_{C^{0}(B_{3}\times B_3)} \leq \delta_0,
\end{equation}
hence
\begin{equation}
\label{eq:grad vic}
\begin{split}
|z\cdot p_{x}+c(z,y)|&\le |z\cdot p_{x}-z\cdot y|+|z\cdot y+c(z,y)|\le 2\delta_0 \qquad \forall\, x, z \in B_{1/7}.
\end{split}
\end{equation}
Thus, if $z \in Q(x,v,h-K_1\sqrt {\omega(\delta_0)})$, by \eqref{eq:vicineeps}, \eqref{gradientivicini}, and \eqref{eq:grad vic} we get
\[
\begin{split}
u(z)&\le v(z)+\omega(\delta_0) \le v(x)+\nabla v (x)\cdot (z-x)+h-K_1\sqrt{\om(\delta_0)}+\omega(\delta_0)\\
&\le u(x)+p_{x}\cdot z -p_{x}\cdot x+h-K_1\sqrt{\om(\delta_0)}+2\om(\delta_0)+2K'\sqrt{ \om(\delta_0)}\\
&\le u(x)-c(z,y)+c(x,y)+h-K_1\sqrt{\om(\delta_0)}+2\om(\delta_0)+2K'\sqrt{ \om(\delta_0)}+4\delta_0\\
&\le u(x)-c(z,y)+c(x,y)+h,
\end{split}
\] 
provided $K_1>0$ is sufficiently large.
This proves the first inclusion, and the second is analogous.

\medskip
\noindent
$\bullet$ {\em Step 3: Both the sections of $u$ and their images are close to ellipsoids with controlled eccentricity,
and $u$ is close to a smooth function near $x_0$.}
We claim that there exists a universal constant \(K_2\geq 1\) such that the following holds: 
For every \(\eta_{0}>0\) small, there exist small positive constants
$h_0=h_{0}(\eta_{0})$ and $\delta_{0}=\delta_{0}(h_{0},\eta_{0})$ such that, for all \(x_{0}\in B_{1/7}\), there is a symmetric matrix \(A\) satisfying 
\begin{equation}\label{uniformA}
\Id/K_2  \le A \le K_2\Id,\qquad \det(A)=1, 
\end{equation}
and such that, for all \( y_{0}\in \pac u(x_{0}) \),
\begin{equation}\label{keyeq}
\begin{split}
&A\left(B_{\sqrt {h_{0}/8}}(x_{0})\right) \subset S(x_{0},y_0,u,h_{0})
\subset A\left(B_{\sqrt {8h_{0}}}(x_{0})\right),\\
& A^{-1}\left(B_{\sqrt {h_{0}/8}}(y_{0})\right) \subset \pac u ( S(x_{0},y_0,u,h_{0}))\subset A^{-1}\left(B_{\sqrt {8h_{0}}}(y_{0})\right).
\end{split}
\end{equation}
Moreover
\begin{equation}\label{eq:parabola2}
\left\|u-C_{x_{0},y_{0}}-\frac{1}{2} \big |A ^{-1}(x-x_{0})\big |^{2}\right\|_{C^{0}\left(A\left(B_{\sqrt {8h_{0}}}(x_{0})\right)\right)}\le \eta_{0}h_{0},
\end{equation} 
where \(C_{x_{0}y_{0}}\)
 is a \(c\)-support function for $u$ at \(x_{0}\), see \eqref{eqCxy0}.

In order to prove the claim,  take \(h_{0}\ll \bar h\) small (to be fixed)  and  \(\delta_{0}\ll h_{0} \)  such that \(K_1\sqrt{\om(\delta_{0})}\le h_{0}/2\),
where $K_1$ is as in Step 2, so that
\begin{equation}\label{eq:seccomparison2}
Q(x_{0},v,h_{0}/2)\subset S(x_{0},y_0,u,h_{0})\subset Q(x_{0},v,3h_{0}/2)\subset \subset B_{1/6} .
\end{equation}
By \eqref{eq:pog} and Taylor formula we get
\begin{equation}\label{taylor}
v(x)=v(x_{0})+\nabla v(x_{0})\cdot (x-x_{0})+\frac 1 2  D^{2}v(x_{0})(x-x_{0}) \cdot (x-x_{0})+O(|x-x_{0}|^{3}),
\end{equation}
so that defining 
\begin{equation}\label{ellipse}
E(x_{0},h_{0}):=\Bigg\{x:\   \frac 1 2 D^{2} v(x_{0})(x-x_{0})\cdot (x-x_{0})\le h_{0} \Bigg\}
\end{equation} 
and using \eqref{eq:pog}, we deduce that, for every \(h_{0}\) universally small,
\begin{equation}\label{sezionivpalle}
E(x_{0},h_{0}/2)\subset Q(x_{0},v,h_{0})\subset E(x_{0}, 2 h_{0}).
\end{equation}
Moreover, always for \(h_{0}\) small, thanks to  \eqref{taylor} and the uniform convexity of \(v\)
\begin{equation}\label{ellisseduale}
\nabla v \big(E(x_{0},h_{0})\big)\subset E^{*}(\nabla v(x_{0}),2h_{0}) \subset \nabla v \big(E(x_{0},3 h_{0}))
\end{equation}
where we have set
\[
 E^{*}(\bar y,h_{0}):= \Bigg\{y:\   \frac 1 2 \big [D^{2} v(\bar y)\big]^{-1}(y-\bar y)\cdot (y-\bar y ) \le  h_{0} \Bigg\}.
\]
By Lemma \ref{sottodif}, \eqref{sezionivpalle}, and \eqref{ellisseduale} applied with $3h_0$ in place of $h_0$, we deduce that for \(\delta_{0}\ll h_{0}\ll \bar h\)
\begin{equation}\label{eq:inclusion}
\pac u\big (S(x_{0},y_0,u,h_{0})\big ) \subset \mathcal{N}_{K''\sqrt{\om(\delta_0)}}\big(\nabla v (E(x_{0},3 h_{0}))\big)\subset E^{*}(\nabla v(x_{0}),7h_{0}).
\end{equation}
Moreover, by  \eqref{gradientivicini}, if $y_0 \in \pac u(x_0)$ and we set $p_{x_0}:=-D_xc(x_0,y_0)$, then
\begin{equation*}
|y_{0}-\nabla v (x_{0})|\le |p_{x_{0}}-\nabla v (x_{0})|+\|D_xc +{\rm Id}\|_{C^{0}(B_{3}\times B_3)}\le K' \sqrt {\omega (\delta_{0}})+\delta_{0}.
\end{equation*}
Thus, choosing \(\delta_{0}\) sufficiently small, it holds
\begin{equation}\label{1}
E^{*}(\nabla v(x_{0}),7h_{0})\subset E^{*}(y_{0},8h_{0}) \quad \forall \,y_{0}\in \pac u(x_{0}).
\end{equation}
We now want to show  that
\begin{equation*}
E^{*}(y_{0},h_{0}/8)\subset \pac u\big (S(x_{0},y_0,u,h_{0})\big )\quad \forall \,y_{0}\in \pac u(x_{0}).
\end{equation*}
Observe that, arguing as above, we get
\begin{equation}\label{2}
E^{*}(y_{0},h_{0}/8)\subset E^{*}(\nabla v(x_{0}),h_{0}/7)\quad \forall \,y_{0}\in \pac u(x_{0})
\end{equation}
provided \(\delta_{0}\) is small enough, so  it is enough to prove that 
\begin{equation*}
E^{*}(\nabla v(x_{0}),h_{0}/7)\subset \pac u\big (S(x_{0},y_0,u,h_{0})\big ).
\end{equation*}
For this, let us define the $c^*$-convex function $u^c:B_3\to \R$ and the convex function $v^*:B_3\to \R$ as
\begin{equation*}
u^{c}(y):=\sup_{x\in B_{1/5}}\big\{-c(x,y)-u(x)\big\},\qquad v^*(y):=\sup_{x\in B_{1/5}}\big\{x\cdot y-v(x)\big\}
\end{equation*}
(see \eqref{eq:dual}). Then it is immediate to check that
\begin{equation}\label{convicine}
|u^{c}-v^{*}|\le \om(\delta_0)+\delta_0 \le 2\om(\delta_0)  \qquad \text{on $B_3$. }
\end{equation}
Also, in view of \eqref{eq:pog}, \(v^{*}\)
is a uniformly convex function of class \(C^{3}\) on the open set \( \nabla v(B_{1/5})\).    
In addition, since 
\begin{equation}\label{ffameno1}
F\subset \pac u (\pacs u^{c}(F))\qquad \text{for any set $F$},
\end{equation}
thanks to \eqref{eq:seccomparison2} and  \eqref{sezionivpalle}  it is enough to show 
\begin{equation}\label{vstar}
\pacs u^{c }(E^{*}(\nabla v(x_{0}),h_{0}/7)) \subset E(x_{0},h_{0}/4).
\end{equation}
For this, we apply Lemma \ref{sottodif} to \(u^{c}\) and \(v^{*}\) to infer
\[
\pacs u^{c }(E^{*}(\nabla v(x_{0}),h_{0}/7))\subset \mathcal N_{K'''\sqrt{\om(\delta )}} \big(\nabla v^{*}(E^{*}(\nabla v(x_{0}),h_{0}/7))\big)\subset  E(x_{0},h_{0}/4),
\]
where we used that 
\[
\nabla v^{*}=[\nabla v]^{-1}\quad \text{and} \quad  D^{2}v^{*}(\nabla v(x_{0}))=[D^{2}v(x_{0})]^{-1}.
\]
Thus, recalling \eqref{eq:inclusion}, we have proved that there exist \(h_{0}\) universally small, and \(\delta_{0}\) small depending on \(h_{0}\), such that 
\begin{equation}\label{codc}
E^{*}(\nabla v (x_{0}),h_{0}/7)\subset \pac u ( S(x_{0},y_0,u,h_{0}))\subset E^{*}(\nabla v (x_{0}), 7 h_{0})
\qquad \forall \,x_{0}\in B_{1/7}.
\end{equation}
Using \eqref{eq:seccomparison2},  \eqref{sezionivpalle},
\eqref{1}, and \eqref{2}, this proves \eqref{keyeq} with \(A:=[D^{2}v(x_{0})]^{-1/2}\).
Also, thanks to \eqref{eq:MAv} and \eqref{eq:pog}, \eqref{uniformA} holds.

In order to prove the second part of the claim, we exploit \eqref{eq:vicineeps}, \eqref{eq:costconv2}, \eqref{eq:yp}, \eqref{gradientivicini},
\eqref{taylor}, and \eqref{uniformA} 
(recall that $C_{x_0,y_0}$ is defined in \eqref{eqCxy0} and that $A=[D^{2}v(x_{0})]^{-1/2}$):
\[
\begin{split}
&\left\|u-C_{x_{0},y_{0}} -\frac{1}{2} \big|A^{-1} (x-x_{0})\big |^{2}\right\|_{C^{0}(E(x_{0},8h_{0}))}\\
&=\left\|u-C_{x_{0},y_{0}}-\frac{1}{2}  D^{2}v(x_{0})(x-x_{0})\cdot(x-x_{0})\right\|_{C^{0}(E(x_{0},8h_{0}))}\\
&\le 2 \|u-v\|_{C^{0}(E(x_{0},8h_{0}))}+\left\|c(x,y_{0})+x\cdot y_0\right\|_{C^{0}(E(x_{0},8h_{0}))}+\left\|c(x_{0},y_{0})+x_{0}\cdot y_0\right \|_{C^{0}(E(x_{0},8h_{0}))}\\
&\quad+\left\|\bigl(y_0-p_{x_0}\bigr)\cdot (x-x_{0})\right \|_{C^{0}(E(x_{0},8h_{0}))}+\left\|\bigl(p_{x_0}-\nabla v (x_{0})\bigr)\cdot (x-x_{0})\right \|_{C^{0}(E(x_{0},8h_{0}))}\\
&\quad +\left\|v-v(x_{0})-\nabla v (x_{0})\cdot (x-x_{0})-\frac{1}{2}  D^{2}v(x_{0})(x-x_{0})\cdot(x-x_{0})\right\|_{C^{0}(E(x_{0},8h_{0}))}\\
&\le 2\omega(\delta_{0})+3\delta_0 + K'\sqrt{\omega(\delta_{0})}+K\left(K_2\sqrt{8h_{0}}\right)^{3}\le \eta_{0}h_{0},
\end{split}
\]
where the last inequality follows by choosing first \(h_{0}\) sufficiently small,  and then \(\delta_{0}\) much smaller than $h_0$.

\medskip
\noindent
$\bullet$ \textit{Step 4: A first change of variables.} Fix \(x_{0} \in B_{1/7}\), \( y_{0}\in \pac u (x_{0})\), define
$M:=-D_{xy}c(x_{0},y_{0})$, and consider the change of variables
\[
\begin{cases}
\bar x :=x-x_{0}\\
\bar y:= M^{-1} (y-y_{0}).
\end{cases}
\]
Notice that, by \eqref{eq:costconv2}, it follows that 
\begin{equation}\label{hessc}
|M-\Id|+|M^{-1}-\Id|\le 3 \delta_0
\end{equation}
for $\delta_0$ sufficiently small.
We also define
\[
\begin{split}
\bar c(\bar x,\bar y)&:=c(x,y)-c(x,y_{0})-c(x_{0},y)+c(x_{0},y_{0}),\\
\bar u(\bar x)&:=u(x)-u(x_{0})+c(x,y_{0})-c(x_{0},y_{0}),\\
\bar u^{\bar c}(\bar y)&:=u^{c }(y)-u^c(y_{0})+c(x_{0}, y)-c(x_{0},y_{0}).
\end{split}
\]
Then \(\bar u\) is \(\bar c\)-convex, $\bar u^{\bar c}$ is $\bar c^*$-convex (where $\bar c^*(\bar y,\bar x)=\bar c(\bar x,\bar y)$), and
\begin{equation}\label{barc}
\bar c(\cdot,\zero)=\bar c(\zero,\cdot)\equiv 0,\qquad D_{\bar x\bar y}\bar c(\zero,\zero)=-\Id.
\end{equation}
We also notice that
\begin{equation}\label{newcsottodif}
\partial _{\bar c }\bar u(\bar x)=M^{-1} \bigl(\pac u(\bar x +x_{0})-y_{0}\bigr).
\end{equation}
Thus, recalling \eqref{keyeq}, and using \eqref{hessc} and \eqref{newcsottodif}, for \(\delta_{0}\) sufficiently small  we obtain
\begin{equation}\label{keyeq2}
\begin{split}
&A\left(B_{\sqrt {h_{0}/9}}\right)  \subset
S(\zero,\zero,\bar u,h_{0})\subset A\left(B_{\sqrt {9h_{0}}}\right),\\
&A^{-1}\left(B_{\sqrt {h_{0}/9}}\right)\subset M^{-1}A^{-1}\left(B_{\sqrt {h_{0}/8}}\right) \subset \partial _{\bar c }\bar u( S(\zero,\zero,\bar u,h_{0})) \subset M^{-1}A^{-1}\left(B_{\sqrt {8h_{0}}}\right)\subset A^{-1}\left(B_{\sqrt {9h_{0}}}\right).
\end{split}
\end{equation}
Since $(T_u)_\sharp f=g$, it follows  that $T_{\bar u}=\bar{\textrm{c}}\textrm{-exp}(\nabla \bar u)$ satisfies
$$
(T_{\bar u})_\sharp \bar f=\bar g, \qquad \text{with}\quad \bar f(\bar x):=f(\bar x+x_0), \quad \bar g(\bar y):=\det(M)\,g(M\bar y+y_0)
$$
(see for instance the footnote in the proof of Theorem \ref{thm:c}).
Notice that, since $|M-\Id|\leq \delta_0$ (by \eqref{eq:costconv2}),  we have $|\det(M)-1| \leq (1+2n)\delta_0$ (for $\delta_0$ small),
so by \eqref{eq:fg} we get
\begin{equation}
\label{eq:barfg}
\|\bar f -\1_{\C_1-x_0}\|_{\infty}+ \|\bar g -\1_{M^{-1}(\C_2-y_0)}\|_{\infty} \leq 2(1+n)\delta_0.
\end{equation}

\medskip
\noindent
$\bullet$ \textit{Step 5: A second change of variables and the iteration argument.} 
We now perform a second change of variable: we set
\begin{equation}
\label{eq:rescaling}
 \begin{cases}
\tilde x:=\frac{1}{\sqrt{h_{0}}} A^{-1}\bar x\\
\tilde y:=\frac{1}{\sqrt{h_{0}}} A\bar y,
\end{cases}
\end{equation}
and define 
\[
c_{1}(\tilde x, \tilde y):=\frac {1}{h_{0}} \bar c\big (\sqrt{h_{0}} A\tilde x,\sqrt{h_{0}} A^{-1}\tilde y\big),
\]
\[
\begin{split}
u_{1}(\tilde x)&:=\frac {1} {h_{0}}\bar u (\sqrt{h_{0}} A \tilde x),\\
u^{c_{1}}_{1}(\tilde y)&:=\frac {1} {h_{0}}\bar u^{\bar c} (\sqrt{h_{0}} A^{-1} \tilde y).
\end{split}
\]
We also define
\[
f_{1}(\tilde x):=\bar f(\sqrt{h_{0}} A \tilde x),\qquad g_{1}(\tilde y):=\bar g(\sqrt{h_{0}}A^{-1} \tilde y). 
\]
Since $\det(A)=1$ (see \eqref{uniformA}), it is easy to check that $(T_{u_{1}})_\sharp f_1=g_1$ (see the footnote in the proof of Theorem \ref{thm:c}).
Also, since \(\left(\|A\|+\|A^{-1}\|\right)\sqrt{h_{0}}\ll 1\), it follows from \eqref{eq:barfg} that
\begin{equation}
\label{eq:fg1}
|f_1-1|+|g_1-1| \leq 2(1+n)\delta_0 \qquad \text{inside $B_3$}.
\end{equation}
Moreover, defining 
\[
\mathcal C^{(1)}_{1}:=  S(\zero,\zero,u_{1},1), \qquad
\mathcal C^{(1)}_{2}:= \partial _{c_{1}} u_{1}( S(\zero,\zero,u_{1},1)),
\]
both $\mathcal C_1^{(1)}$ and \(\mathcal C^{(1)}_{2}\) are closed, and thanks to \eqref{keyeq2}
\begin{equation}\label{keyeq3}
B_{1/3} \subset \mathcal C^{(1)}_{1},\mathcal C^{(1)}_{2}\subset B_{3}.
\end{equation}
Also, since  $(T_{u_{1}})_\sharp f_1=g_1$, arguing as in the proof of Theorem \ref{thm:c} 
we get
\[
(T_{u_{1}})_{\sharp}\big(f_{1}\1_{\mathcal C^{(1)}_{1}}\big)=\big(g_{1}\1_{\mathcal C^{(1)}_{2}}\big),
\]
and by \eqref{eq:fg1}
$$
\|f_1\1_{\mathcal C^{(1)}_{1}}-\1_{\mathcal C^{(1)}_{1}}\|_{\infty}+\|g_1\1_{\mathcal C^{(1)}_{2}}-\1_{\mathcal C^{(1)}_{2}}\|_{\infty} \leq 2(1+n)\delta_0.
$$
Finally, using \eqref{barc}, \eqref{eq:costconv2}, and \eqref{eq:parabola2}, it is  easy to check that 
$$
\|c_{1}(\tilde x, \tilde y)+\tilde x\cdot \tilde y\|_{C^{2,\alpha}(B_{3} \times B_{3})}\le \delta_{0},
\qquad
\left\|u_{1}-\frac {1}{2}|\tilde x|^{2}\right\|_{C^{0}(B_{3})}\le \eta_{0}
$$ provided $h_0$ is small enough. Indeed, the second inequality is just a direct consequence of \eqref{eq:parabola2}, while  \eqref{barc} and a Taylor expansion yield
\begin{equation*}\label{e:taylorholder}
c_{1}(\tilde x, \tilde y)=\frac {1}{h_{0}} \bar c\big (\sqrt{h_{0}} A\tilde x,\sqrt{h_{0}} A^{-1}\tilde y\big)=\tilde x\cdot \tilde y +O(\delta_0 K_2^2 (K_2\sqrt{h_0})^{\alpha}).
\end{equation*}
so we can choose \(h_0\) small enough to ensure that
\begin{equation}\label{e:h0}
O(\delta_0 K_2^2 (K_2\sqrt{h_0})^{\alpha})\le \delta_0.
\end{equation}
This shows that $u_{1}$ satisfies the same assumptions as $u$ with $\delta_0$ replaced by $2(1+n)\delta_0$.
Hence, up to take $\delta_0$ slightly smaller,
we can apply Step 3 to $u_{1}$, and we find a symmetric  matrix \(A_{1}\) satisfying 
\begin{equation*}
\Id/K_2  \le A_{1} \le K_2\Id, \qquad \det(A_1)=1,
\end{equation*}
\begin{equation*}
\begin{split}
&A_1\left(B_{\sqrt {h_{0}/8}}\right) \subset S(\zero,\zero,u_{1},h_{0})\subset A_1\left(B_{\sqrt {8h_{0}}}\right),\\
& A_{1}^{-1}\left(B_{\sqrt {h_{0}/8}}\right) \subset \partial_{c_{1}} u_{1} ( S(\zero,\zero,u_{1},h_{0})) \subset A_{1}^{-1}\left(B_{\sqrt {8h_{0}}}\right),
\end{split}
\end{equation*}
\begin{equation*}
\left\|u_{1}-\frac{1}{2} \big |A_{1} ^{-1}\tilde x\big |^{2}\right\|_{C^{0}(A_{1}(B(0,\sqrt {8h_{0}}))}\le \eta_{0}h_{0}.
\end{equation*} 
(Here \(K_{2}\) and \(h_{0}\) are as in Step 3.)

This allows us to apply to $u_1$ the very same construction
as the one  used above
to define $u_1$ from $\bar u$: we set
\[
\begin{split}
c_{2}(\tilde x, \tilde y):=\frac {1}{h_{0}} c_{1}\big (\sqrt{h_0} A_{1}\tilde x,\sqrt{h_{0}} A_{1}^{-1}\tilde y\big),\qquad
u_{2}(\tilde x)&:=\frac {1} {h_{0}}u_{1} (\sqrt{h_{0}} A_{1} \tilde x),\\
\end{split}
\]
so that $(T_{u_2})_\sharp f_2=g_2$ with
$$
f_{2}(\tilde x):=f_1(\sqrt{h_{0}} A_1 \tilde x),\qquad g_{2}(\tilde y):=\bar g(\sqrt{h_{0}}A_1^{-1} \tilde y).
$$
Arguing as before, it is easy to check that $u_2,c_2,f_2,g_2$ satisfy the same assumptions as
$u_1,c_1,f_1,g_1$ with exactly the same constants, provided \(h_0\ll1\) satisfies \eqref{e:h0}.

So we can keep iterating this construction, defining for any
\(k\in \N\) 
\[
\begin{split}
c_{k+1}(\tilde x, \tilde y):=\frac {1}{h_{0}} c_{k}\big (\sqrt{h_0} A_{k}\tilde x,\sqrt{h_{0}} A_{k}^{-1}\tilde y\big),\qquad
u_{k+1}(\tilde x)&:=\frac {1} {h_{0}}u_{k} (\sqrt{h_{0}} A_{k} \tilde x),\\
\end{split}
\]
where $A_k$ is the matrix constructed in the $k$-th iteration.
In this way, if we set
$$
M_k:=A_k \cdot \ldots \cdot A_1, \qquad \forall\, k\geq 1,
$$
we obtain a sequence of symmetric matrices satisfying
\begin{equation}\label{akappa}
\Id/K^{k}_2  \le M_{k} \le K^{k}_2\Id, \qquad \det(M_k)=1,
\end{equation}
and such that
\begin{equation}\label{eq:parabola3}
M_{k}\left(B_{(h_0/8)^{k/2}}\right) \subset S(\zero,\zero,u_{k},h^{k}_{0})\subset M_{k}\left(B_{(8h_0)^{k/2}}\right).
\end{equation}

\medskip
\noindent
$\bullet$ \textit{Step 6: \(C^{1,\beta}\) regularity}. We now show that, for any $\beta \in (0,1)$,
we can choose $h_0$ and $\delta_0=\delta_0(h_0)$ small enough so that $u_1$ is \(C^{1,\beta}\) at the origin (here $u_1$ is the function constructed in the previous step).
This will imply that $u$ is $C^{1,\beta}$ at $x_0$ with universal bounds, which by the arbitrariness of $x_0 \in B_{1/7}$
gives \(u\in C^{1,\beta}(B_{1/7})\).

Fix \(\beta \in (0,1)\). Then
by \eqref{akappa} and  \eqref{eq:parabola3}
we get
\begin{equation}\label{sezionifighe}
B_{\bigl(\sqrt{h_0}/(\sqrt{8}K_2)\bigr)^{k}}\subset S(\zero,\zero,u_{1},h^{k}_{0})\subset B_{\bigl(K_2\sqrt{8h_0}\bigr)^{k}},
\end{equation}
so defining $r_0:=\sqrt{h_0}/(\sqrt{8}K_2)$ we  obtain
\[
\|u_{1}\|_{C^{0}(B_{r^{k}_{0}})}\le h_0^k =\left(\sqrt{8}K_2r_0\right)^{2k} \leq r_0^{(1+\beta)k},
\]
provided \(h_{0}\) (and so $r_0$) is sufficiently small.
This implies the \(C^{1,\beta}\) regularity of \(u_{1}\) at $\zero$, concluding the proof.
\end{proof}

\begin{remark}[Local to global principle]\label{localtoglobal}
If \(u\) is differentiable at \(x\) and \(c\) satisfies {\bf (C0)}-{\bf (C1)}, then every ``local support'' at \(x\)  is also a ``global $c$-support'' at \(x\), that is,
$\p_cu(x)=\cexp_x(\p^-u(x))$.
To see this, just notice that
\[
\emptyset \ne \pac u (x)\subset \cexp_x(\p^{-} u (x))=\{\cexp_x(\nabla u (x))\}
\]
(recall \eqref{eq:rel subdiff2}), so necessarily the two sets have to coincide.
\end{remark}

\begin{corollary}\label{cconv} Let \(u\) be as in Theorem \ref{c1alpha}. Then \(u\) is strictly \(c\)-convex in \(B_{1/7}\).
More precisely, for every \(\gamma>2\) there exist \(\eta_{0},\delta_{0}>0\) depending only on \(\gamma\) such that, if the hypotheses of Theorem \ref{c1alpha} are satisfied,
then, for all \(x_{0}\in B_{1/7}\), $y_0 \in \p_cu(x_0)$, and \(C_{x_{0},y_{0}}\) as in \eqref{eqCxy0}, we have
\begin{equation}\label{strict}
\inf_{\partial B_r(x_0)}\big\{ u-C_{x_{0},y_{0}}\big\}\ge c_0r^{\gamma}\qquad \forall \,r\le \dist(x_{0},\partial B_{1/7}),
\end{equation}
with $c_0>0$ universal.
\end{corollary}

\begin{proof}With the same notation as in the proof of Theorem \ref{c1alpha}, it is enough to show that
\[
\inf_{\partial B_r}u_{1}\ge r^{1/\beta},
\]
where \(u_{1}\) is the function constructed in Step 5 of the proof of Theorem \ref{c1alpha}.
Defining \(\varrho_{0}:=K_2\sqrt{8h_0}\), it follows from \eqref{sezionifighe} that
\[
\inf_{\partial B_{\varrho^{k}_{0}}} u_{1} \ge h_{0}^{k}=\left(\varrho_{0}/(\sqrt{8}K_{2})\right)^{2k}\ge \varrho_{0}^{\gamma k},
\]
provided \(h_{0}\) is small enough.
\end{proof}
A simple consequence of the above results is the following:
\begin{corollary}\label{open}Let \(u\) be as in Theorem \ref{c1alpha}, then \(T_{u}(B_{1/7})\) is open.
\end{corollary}

\begin{proof}Since \(u\in C^{1,\beta}(B_{1/7})\) we have that \(T_{u}(B_{1/7})=\pac u(B_{1/7})\) (see
Remark \ref{localtoglobal}).
We claim that it is enough to show that if \(y_{0}\in \pac u(B_{1/7})\), then  there exists \(\e=\e(y_0)>0\) small such that, for all \(|y-y_{0}|<\e\),
the function \(u(\cdot)+c(\cdot,y)\) has a local minimum at some point \(\bar x\in B_{1/7}\). Indeed, if this is the case, then
\[
\nabla u (\bar x)=-D_{x}c(\bar x, y),
\]
and so \(y\in \partial_{c} u(\bar x)\)  (by Remark \ref{localtoglobal}), hence $B_\e(y_0)\subset T_{u}(B_{1/7})$.

To prove the above fact, fix \(r>0\) such that \( B_{r}( x_{0})\subset B_{1/7}\), and pick $\bar x$ a point in $\overline B_{r}( x_{0})$ where the function $u(\cdot)+c(\cdot,y)$ attains its minimum, i.e.,
\[
\bar x \in \argmin_{\overline B_{r}( x_{0})}\big\{u(x)+c(x,y)\big\}.
\]
Since, by \eqref{strict},
\[
\begin{split}
\min_{x \in \partial B_{r}( x_{0})}\big \{u(x)+c(x,y)\big\}&\ge \min_{x \in \partial B_{r}( x_{0})}\big \{u(x)+c(x,y_{0})\big\}-\e \|c\|_{C^{1}} \\
&\ge u(x_{0})+c(x_{0},y_{0})+ c_{0}r^{\gamma}-\e\|c\|_{C^{1}} ,
\end{split}
\]
while
\[
u(x_{0})+c(x_{0},y)\le c(x_{0},y_{0})+u(x_{0})+\e\|c\|_{C^{1}},
\]
choosing \(\e< \frac{c_{0}}{2\|c\|_{C^{1}}}r^{\gamma}\) we obtain that \(\bar x\in B_{r}(x_{0})\subset B_{1/7}\). This implies that $\bar x$ is a local minimum for \(u(\cdot)+c(\cdot,y)\), concluding the proof.
\end{proof}

\section{Comparison principle and \(C^{2,\alpha}\) regularity}\label{sect:C2a}

We begin this section with a change of variable formula for the $c$-exponential map.

\begin{lemma}\label{c2smooth}Let $\Omega$ be an open set,  \(v\in C^{2}(\Omega)\), and assume that \(\nabla v(\Omega)\subset {\rm Dom}\cexp\) and that
\[
D^{2}v(x)+D_{xx}c\big(x,\cexp_x(\nabla v (x))\big )\ge 0 \quad \forall\, x \in \Omega.
\]
Then, for every Borel set \(A\subset \Omega\),
\begin{equation*}
|\cexp(\nabla v (A))|\le \int_{A} \frac{\det\big(D^{2}v(x)+D_{xx}c\bigl(x,\cexp_x(\nabla v(x))\bigr)\big)}{\bigl|\det \bigl(D_{xy}c\bigl(x,\cexp_x(\nabla v(x))\bigr)\bigr)\bigr|}\, dx.
\end{equation*}
In addition, if the map \(x\mapsto \cexp_x(\nabla v(x))\) is injective, then equality holds.
\end{lemma}

\begin{proof}
The result follows from a direct application of the Area Formula \cite[Section 3.3.2, Theorem 1]{EG} once one notices that, differentiating the identity
\[
\nabla v(x)=-D_{x}c\big(x,\cexp_x(\nabla v(x))\big )
\]
 (see \eqref{eq:cexp}), the Jacobian determinant of the $C^1$ map \(x\mapsto \cexp_x(\nabla v(x))\) is given precisely by
\[
\frac{\det\big(D^{2}v(x)+D_{xx}c\bigl(x,\cexp_x(\nabla v(x))\bigr)\big)}{\bigl|\det \bigl(D_{xy}c\bigl(x,\cexp_x(\nabla v(x))\bigr)\bigr)\bigr|}.
\]
\end{proof}

In the next proposition we show a comparison principle between $C^1$ $c$-convex functions and 
smooth solutions to the Monge-Amp\`ere equation. \footnote{A similar result
for the case $c(x,y)=|x-y|^p$
appeared in \cite[Theorem 6.2]{CMN}.
Here, however, we have to deal with some additional difficulties
due to the fact that the $c$-exponential map is not necessarily defined on the whole $\R^n$.}
As already mentioned at the beginning of  Section \ref{sect:C1a} (see also Remark \ref{localtoglobal}),
the $C^1$ regularity of $u$ is crucial to ensure that the $c$-subdifferential
coincides with its local counterpart $\cexp(\p^-u)$.
 
Here and in the sequel, we 
use $\co[E]$ to denote the convex hull of a set $E$. Also,
recall that $ \mathcal N_r(E)$ denotes the $r$-neighborhood of $E$.

\begin{proposition}[Comparison principle]\label{comparison}
Let $u$ be a \(c\)-convex function of class $C^1$ inside the set $S:=\{u < 1\}$, and assume that $u(\zero)=0$, $B_{1/K}\subset S \subset B_K$,
and that \(\nabla u(S)\subset\subset {\rm Dom}\cexp\).
Let $f,g$ be two densities such that 
\begin{equation}\label{eq:fg eps}
\|f/\l_1-  1 \|_{C^0(S)}+ \| g/\l_2- 1\|_{C^0(T_u(S))} \leq \e
\end{equation}
for some constants $\lambda_1,\lambda_2 \in (1/2,2)$ and $\e\in(0,1/4)$,
and assume that $(T_{u})_\sharp f =g$.
Furthermore, suppose that
\begin{equation}\label{eq:c delta}
\|c + x\cdot y\|_{C^2(B_K\times B_K)} \leq \delta.
\end{equation}
Then there exist a universal constant $\gamma \in (0,1)$, and \(\delta_{1}=\delta_{1}(K)>0\) small, such that the following holds:
Let $v$ be the solution of
$$
\begin{cases}
\det(D^2v)=\l_1/\l_2\quad &\text{in }\mathcal N_{\delta^\gamma} ( \co[S]),\\
v=1 & \text{on }\p\bigl(\mathcal N_{\delta^\gamma} ( \co[S])\bigr).
\end{cases}\
$$
Then
\begin{equation} \label{figurati}
\|u-v\|_{C^0(S)} \leq C_K\left( \e +\delta^{\gamma/n}\right)\qquad \text{provided $\delta \leq \delta_1$,}
\end{equation}
where $C_K$ is a constant independent of $\lambda_1$, $\lambda_2$, $\e$, and $\delta$ (but which depends on $K$).
\end{proposition}

\begin{proof}
First of all we observe that,
since $u(\zero)=0$, $u=1$ on $\partial S$, $S \subset B_K$, and $\|c + x\cdot y\|_{C^2(B_K)} \leq \delta \ll 1$,
it is easy to check that there exists a universal constant $a_1>0$ such that
\begin{equation}
\label{eq:gradient nonzero}
|D_xc(x,y)| \geq a_1\qquad \forall\, x\in \partial S,\, y =\cexp_x(\nabla u(x)).
\end{equation}
Thanks to \eqref{eq:gradient nonzero} and \eqref{eq:c delta}, it follows from the Implicit Function Theorem that, for each $x\in \partial S$, the boundary of the set
$$
E_x:=\big\{z \in B_K:\ c(z,y)-c(x,y)+u(x) \leq 1\big\}
$$
is of class $C^2$ inside $B_K$, and its second fundamental form is bounded by $C_K\delta$, where $C_K>0$ depends only on $K$.
Hence, since $S$ can be written as
$$
S:=\bigcap_{x\in \partial S} E_x,
$$
it follows that
$$
\text{$S$ is a $(C_K\delta)$-semiconvex set},
$$
that is, for any couple of points $x_0,x_1 \in S$ the ball centered at $x_{1/2}:=(x_0+x_1)/2$ of radius $C_K\delta|x_1-x_0|^2$ intersects $S$. Since $S\subset B_K$, this implies that
$\co[S] \subset \mathcal N_{C'_{K}\delta} (S)$ for some positive constant $C_K'$
depending only on $K$.
Thus,  for any $\gamma \in(0,1)$ we obtain 
\[
\mathcal N_{\delta^\gamma} ( \co[S]) \subset \mathcal N_{(1+C'_{K}) \delta^\gamma}( S).
\]
Since $v=1$ on $\p\bigl(\mathcal N_{\delta^\gamma} ( \co[S])\bigr)$ and $\l_1/\l_2\in (1/4,4)$,
by standard interior estimates for solution of the  Monge-Amp\`ere equation with constant right hand side (see for instance \cite[Lemma 1.1]{CL}), we obtain
\begin{eqnarray}
&\osc_S v \le C''_{K} \label{osc}\\
&1- C''_{K}\delta^{\gamma/n}\le v< 1 &\quad \text{on $\partial S$}, \label{MAsup}\\
&D^2 v \geq \delta^{\gamma/\tau} \Id/C''_{K}&\quad \text{in \(\co[S]\)}, \label{MAder}
\end{eqnarray}
for some $\tau>0$ universal, and some constant $C''_{K}$ depending only on $K$.

Let us define
$$
v^+:=(1+4\e+2\sqrt{\delta} )v-4\e-2\sqrt{\delta},\qquad v^-:=(1 -4 \e-\sqrt{\delta}/2)v+4\e+\sqrt{\delta}/2+2 C''_{K} \delta^{\gamma/n}.
$$
Our goal is to show that we can choose \(\gamma\) universally small
so that $v^- \geq u \geq v^{+}$ on \(S\). Indeed, if we can do so, then
by  \eqref{osc} this will imply \eqref{figurati}, concluding the proof.

First of all notice that, thanks to \eqref{MAsup}, $v^- > u > v^{+}$ on \(\partial S\).
Let us show first  that \(v^{+}\le v\).

Assume by contradiction this is not the case.
Then, since $u > v^{+}$ on \(\partial S\),  
\[
\emptyset \ne Z:=\big\{ u< v^{+}\big\} \subset \subset S.
\]
Since \(v^{+}\) is convex, taking any  supporting plane to \(v^{+}\) at \(x\in Z\), moving it down and then lifting  it up until it touches \(u\) from below, we deduce that
\begin{equation}\label{inclusionsottodif}
\n v^{+}(Z)\subset \nabla u (Z)
\end{equation}
(recall that both $u$ and $v^+$ are of class $C^1$),
thus by Remark \ref{localtoglobal}
\begin{equation}\label{bologna}
|\cexp (\nabla v^{+}(Z))|\le |T_{u}(Z)|.
\end{equation}
We show that this is impossible.
 For this, using \eqref{MAder} and choosing \(\gamma:= \tau/4\),  for any \(x\in Z\) we compute
\[
\begin{split}
D^{2}v^{+}(x)+D_{xx}c\bigl(x,\cexp_x(\nabla v^{+}(x))\bigr)&\ge(1+\sqrt{\delta}+4\e)D^{2}v+\sqrt{\delta}D^{2}v-\delta \Id\\
&\ge(1+\sqrt{\delta}+4\e)D^{2}v+(\delta^{3/4}/C''_{K}-\delta) \Id\\
&\ge(1+\sqrt{\delta}+4\e)D^{2}v,
\end{split}
\]
provided \(\delta\) is sufficiently small, the smallness depending only on  \(K\). Thus, thanks \eqref{eq:c delta} we have
\[
\begin{split}
\frac{\det\big(D^{2}v^{+}(x)+D_{xx}c(x,\cexp_x(\nabla v^{+}(x)))\big)}{\bigl|\det \bigl(D_{xy}c\bigl(x,\cexp_x(\nabla v^{+}(x))\bigr)\bigr)\bigr|}&\ge
\frac{\det\big((1+\sqrt{\delta}+4\e)D^{2}v\big)}{1+\delta}\\
&\ge (1+\sqrt{\delta}+4\e)^{n}(1-2\delta)\frac{\l_1}{\l_2}\\
&\ge (1+4n\e)\frac{\l_1}{\l_2}.
\end{split}
\]
In addition, thanks \eqref{MAder} and \eqref{eq:c delta}, since $\delta^{\gamma/\tau}=\delta^{1/4} \gg \delta$
we see that
$$
D^2 v^+ > \|D_{xx}c\|_{C^0(B_K\times B_K)} \Id \qquad \text{inside $\co[S].$}
$$
Hence,
for any $x,z \in Z$, \(x\ne z\) and $y=\cexp_x(\nabla v^+(x))$ (notice that $\cexp_x(\nabla v^+(x))$ is well-defined because of \eqref{inclusionsottodif}
and the assumption \(\nabla u(S)\subset\subset {\rm Dom}\cexp\)), it follows
\[
\begin{split}
v^+(z)+c(z,y) &\geq v^+(x)+c(x,y)\\
&\qquad \qquad \quad+ \frac{1}{2}\int_0^1 \Bigl(D^2 v^+\bigl(tz+(1-t)x\bigr)+D_{xx}c\bigl(tz+(1-t)x,y\bigr)\Bigr)[z-x,z-x]\,dt\\
&> v^+(x)+c(x,y),
\end{split}
\]
where we used that $\nabla v^+(x)+D_xc(x,y)=0$.
This means that the supporting function $z \mapsto -c(z,y)+c(x,y)+v^+(x)$ can only touch $v^+$ from below at $x$, which implies that
the map $Z \ni x\mapsto \cexp_x(\nabla v^+(x))$ is injective.
Thus, by Lemma \ref{c2smooth} we get 
\begin{equation}
\label{eq:measure vZ}
|\cexp (\nabla v^{+}(Z))| \geq (1+4n\e)\frac{\l_1}{\l_2}\,|Z|.
\end{equation}
On the other hand, since $u$ is $C^1$, it follows from $(T_u)_\sharp f=g$ and \eqref{eq:fg eps} that
$$
|T_u(Z)|=\int_Z \frac{f(x)}{g(T_u(x))}\,dx \leq \frac{\l_1(1+\e)}{\l_2(1-\e)}\,|Z|\leq (1+3\e)\frac{\l_1}{\l_2}\,|Z|.
$$
This estimate combined with \eqref{eq:measure vZ} shows that \eqref{bologna} is impossible unless $Z$ is empty. This proves that $v^+\leq u$.

The proof of the inequality \(v^{-}\le u\) follows by the same argument except for a minor modification.
More precisely, let us assume by contradiction that $W:=\{ u>v^{-}\}$
is nonempty. In order to apply the previous argument we would need to know that \(\nabla v^{-}(W)\subset {\rm Dom}\cexp\).
However, since the gradient of $v$ can be very large near $\partial S$, this may be a problem.

To circumvent this issue we argue as follows: since $W$ is nonempty, there exists a positive constant $\bar \mu$ such that 
$u$ touches $v^{-}+ \bar \mu$ from below inside $S$.
Let $E$ be the contact set, i.e., $E:=\{ u=v^{-}+ \bar\mu\}$.
Since both $u$ and $v^-$ are $C^1$, $\nabla u=\nabla v^-$ on $E$.
Thus, if $\eta>0$ is small enough, then the set $W_\eta:=\{ u>v^{-}+ \bar\mu-\eta\}$ is nonempty
and \(\nabla v^{-}(W_{\eta})\) is contained in a small neighborhood of \(\nabla u(W_{\eta})\), which is compactly contained in \( {\rm Dom}\cexp\).
At this point, one argues exactly as in the first part of the proof, with $W_\eta$ in place of $Z$, to find a contradiction.
\end{proof}

\begin{theorem}\label{C2alpha} Let \(u,f,g,\eta_0, \delta_0\) be as in Theorem \ref{c1alpha}, and assume in addition that \(c\in C^{k,\alpha}(B_{3}\times B_{3})\) and $f,g \in C^{k,\alpha}(B_{1/3})$
for some $k \geq 0$ and $\alpha \in (0,1)$.
There exist small constants \(\eta_{1} \leq \eta_0\) and \(\delta_{1}\leq \delta_0\) such that, if
 \begin{equation}\label{eq:fgC2a}
 \|f-  \1_{\C_{1}} \|_{\infty}+ \| g- \1_{\C_{2}}\|_{\infty} \leq \delta_{1},
 \end{equation}
 \begin{equation}\label{eq:costconvC2a}
 \|c(x,y)+x\cdot y\|_{C^{2,\alpha}(B_{3}\times B_{3})}\le \delta_{1},
 \end{equation}
 and 
\begin{equation}\label{eq:parabolaC2a}
\left\|u-\frac{1}{2} |x|^{2}\right\|_{C^{0}(B_{3})}\le \eta_{1},
\end{equation} 
then \(u \in C^{k+2,\alpha}(B_{1/9})\).
\end{theorem}

\begin{proof} We divide the proof in two steps.

\medskip
\noindent
$\bullet$ \textit{Step 1: \(C^{1,1}\) regularity}.
Fix a point $x_0 \in B_{1/8}$, and set $y_0:=\cexp_{x_0}(\n u(x_0))$.
Up to replace $u$ (resp. $c$) with the function $u_1$ (resp. $c_1$) constructed in 
Steps 4 and 5 in the proof of Theorem \ref{c1alpha}, we can assume that \(u\ge 0\),  \(u(\zero)=0\), that
\[
S_{h}:=S(\zero,\zero,u,h)=\{u\le h\},
\]
and that
\begin{equation}\label{idpunto}
D_{xy}c(\zero,\zero)=-\Id.
\end{equation}
Under these assumptions we will show that  the sections of \(u\)  are of ``good shape'', i.e.,
\begin{equation}\label{goodshape}
B_{\sqrt{h}/K}\subset S_{h} \subset B_{K\sqrt{h}} \qquad \forall\,h\leq h_{1},
\end{equation}
for some universal \(h_{1}\) and \(K\). Arguing as in Step 6 of Theorem \ref{c1alpha}, this will give that \(u\) is \(C^{1,1}\) at the origin, and thus at every point in \(B_{1/8}\).

First of all notice that, thanks to \eqref{eq:parabolaC2a}, for any \(h_{1}>0\) we can choose \(\eta_{1}=\eta_1(h_1)>0\) small enough such that \eqref{goodshape} holds for \(S_{h_{1}}\)
with $K=2$. Hence, assuming without loss of generality that $\delta_1 \leq 1$, 
we see that 
$$
B_{\sqrt{h_{1}}/3}\subset \mathcal N_{\delta_1^\gamma\sqrt{h_{1}}}(\co[S_{h_{1}}])\subset B_{3\sqrt{h_{1}}},
$$
where \(\gamma\) is the exponent from Proposition \ref{comparison}.
Let $v_1$ solve the Monge-Amp\`ere equation
$$
\left\{
\begin{array}{ll}
\det(D^2v_{1})=f(\zero)/g(\zero) &\text{in }\mathcal N_{\delta_1^\gamma\sqrt{h_{1}}}(\co[S_{h_{1}}]),\\
v_{1}=h_{1} & \text{on }\partial\mathcal N_{\delta_1^\gamma\sqrt{h_{1}}}(\co[S_{h_{1}}]).
\end{array}
\right.
$$
Since $B_{1/3}\subset N_{\delta_1^\gamma\sqrt{h_{1}}}(\co[S_{h_{1}}])/\sqrt{h_1} \subset B_3$,
by standard Pogorelov estimates applied to the function $v_1(\sqrt{h_1}x)/h_1$ (see for instance \cite[Theorem 4.2.1]{G}), it follows that 
$|D^2v_1(0)|\leq M$, with $M>0$ some large universal constant. 

 Let \(h_{k}:=h_{1}2^{-k}\) and define $\bar K\geq 3$ to be the largest number such that any solution $w$ of 
\begin{equation}
\label{MA1}
\left\{
\begin{array}{ll}
\det(D^2w)=f(\zero)/g(\zero) &\text{in }Z,\\
w=1 & \text{on }\partial Z,
\end{array}
\right.
\qquad \text{with}\quad B_{1/\bar K}\subset Z\subset B_{\bar K},
\end{equation}
satisfies $|D^2w(0)|\leq M+1$. \footnote{The fact that $\bar K$ is well defined (i.e., $3 \leq \bar K <\infty$) follows by the following facts: first of all,
by definition, $M$ is an a-priori bound for $|D^2w(0)|$ whenever $w$ is a solution of \eqref{MA1} with $B_{1/3}\subset Z \subset B_3$, so $\bar K \geq 3$.
On the other hand  $\bar K \leq \sqrt{2(M+1)}$. Indeed, since \(1/2\le f(\zero)/g(\zero)\le2\) (by  \eqref{eq:fgC2a}) and \(M\ge 1\),  the function
$$\bar w:=(M+1)x_1^2 +\frac{f(\zero)}{g(\zero)}\,\frac{x_2^2}{M+1} + x_3^2 + \ldots + x_n^2$$
is a solution of \eqref{MA1} such that $B_{1/\sqrt{2(M+1})}\subset B_{1/\sqrt{M+1}}\subset \{\bar w\le1\} \subset B_{\sqrt{2(M+1)}}$ and $|D^2\bar w(0)|=2(M+1)$ .}
We prove by induction that \eqref{goodshape} holds with \(K=\bar K\).

If $h=h_1$ then we already know that \eqref{goodshape} holds with $K=2$ (and so with $K=\bar K$).

Assume now that \eqref{goodshape} holds with
\(h=h_{k}\) and \(K=\bar K\), and we want to show that it holds with \(h=h_{k+1}\). For this, for any $k \in \N$ we consider $u_k$
the solution of 
$$
\left\{
\begin{array}{ll}
\det(D^2v_k)=f(\zero)/g(\zero) &\text{in }\mathcal N_{\delta_k^\gamma\sqrt{h_{k}}}(\co[S_{h_{k}}]),\\
v_k=h_{1}2^{-k} & \text{on }\partial N_{\delta_k^\gamma\sqrt{h_{k}}}(\co[S_{h_{k}}]),
\end{array}
\right.
$$
where
\[
\delta_{k}:=\|c(x,y)+x\cdot y\|_{C^{2}(S_{h_{k}}\times T_{u}(S_{h_{k}}))} \leq \delta_1.
\]
Let us consider the rescaled functions
\[
 \bar u_{k}(x):=u\big(\sqrt{h_{k}}x\big)/h_{k}, \qquad  \bar v_{k}(x):=v_{k}\big(\sqrt{h_{k}}x\big)/h_{k}.
\] 
Since by the inductive hypothesis $B_{1/\bar K}\subset \bar S_k:=\{\bar u_{k} \leq 1 \} \subset B_{\bar K}$,
we can apply Proposition \ref{comparison} to deduce that
\begin{equation}
\label{eq:close}
\|\bar u_k -\bar v_k\|_{C^0(\bar S_k)} \leq C_{\bar K}\Big(\osc_{S_{h_{k}}} f+ \osc_{T_u(S_{h_{k}})} g + \delta_k^{\gamma/n}\Big) \leq C_{\bar K}(\delta_1+\delta_1^{\gamma/n}).
\end{equation}
This implies in particular that, if $\delta_1$ is sufficiently small, $B_{1/(2\bar K)}\subset \{\bar v_{k} \leq 1 \} \subset B_{2\bar K}$.
By standard estimates on the sections of solutions to the Monge-Amp\`ere equation,
the shapes of  $\{\bar v_{k} \leq 1 \}$ and $\{\bar v_k\leq 1/2\}$ are comparable,
and in addition sections are well included into each other \cite[Theorem 3.3.8]{G}:
there exists a universal constant $L>1$ such that
$$
B_{1/(L\bar K)} \subset \{\bar v_k\leq 1/2\} \subset B_{L\bar K},\qquad \dist\bigl(\{\bar v_k\leq 1/4\},\partial \{\bar v_k\leq 1/2\}\bigr) \geq 1/(LK).
$$
Using again \eqref{eq:close} we deduce that, if $\delta_1$ is sufficiently small,
$$
B_{1/(2L\bar K)} \subset \{\bar u_k\leq 1/2\} \subset B_{2L\bar K},
\qquad \dist\bigl(\{\bar u_k\leq 1/4\},\partial \{\bar u_k\leq 1/2\}\bigr) \geq 1/(2LK)
$$
so, by scaling back,
\begin{equation}\label{nottoobad}
B_{\sqrt{h_{k+1}}/(2L\bar K)}\subset S_{h_{k+1}}\subset B_{2L\bar K\sqrt{h_{k+1}}}, \qquad \dist\bigl( S_{h_{k+2}},\partial  S_{h_{k+1}}\bigr) \geq \sqrt{h_k}/(2LK).
\end{equation}
This allows us to apply Proposition \ref{comparison} also to $\bar u_{k+1}$ to get
\begin{equation}
\label{eq:close1}
\|\bar u_{k+1} -\bar v_{k+1}\|_{C^0(\bar S_{k+1})} \leq C_{2L\bar K}\Big(\osc_{S_{h_{k+1}}} f+ \osc_{T_u(S_{h_{k+1}})} g + \delta_{k+1}^{\gamma/n}\Big).
\end{equation}
We now observe that, by \eqref{goodshape} and the $C^{1,\beta}$ regularity of $u$ (see Theorem \ref{c1alpha}),
it follows that
$$
\diam(S_{h_{k}})+ \diam (T_u(S_{h_{k}})) \leq C h_k^{\beta/2},
$$
so by the $C^{0,\alpha}$ regularity of $f$ and $g$,
and the $C^{2,\alpha}$ regularity of $c$, we have (recall that $\gamma <1$)
\begin{equation}
\label{sigma}
\osc_{S_{h_{k}}} f+ \osc_{T_u(S_{h_{k}})} g + \delta_k^{\gamma/n}\le C'h_{k}^{\sigma}, \qquad \sigma:=\frac{\alpha \beta\gamma}{2n}
\end{equation}
Hence, by \eqref{eq:close} and \eqref{eq:close1},
$$
\|\bar u_k -\bar v_k\|_{C^0(\bar S_k)}+ \|\bar u_{k+1} -\bar v_{k+1}\|_{C^0(\bar S_{k+1})} \leq C \left(C_{\bar K} + C_{2L\bar K}\right)h_{k}^{\sigma} ,
$$
from which we deduce (recall that $h_k=2h_{k+1}$)
\[
\begin{split}
\|v_{k}-v_{k+1}\|_{C^{0}(S_{h_{k+1}})}&\le \|v_{k}-u\|_{C^{0}(S_{h_{k}})}+\|u-v_{k+1}\|_{C^{0}(S_{h_{k+1}})}\\
&= h_{k}\|\bar u_{k}-\bar v_{k}\|_{C^{0}(S_{k})}+h_{k+1}\|\bar u_{k+1}-\bar v_{k+1}\|_{C^{0}(S_{k+1})}\\
&\le C \left(C_{\bar K} + C_{2L\bar K}\right)h_k^{1+\sigma}. 
\end{split}
\]
Since $v_k$ and $v_{k+1}$ are two strictly convex solutions of the Monge Amp\`ere equation with constant right hand side inside \(S_{h_{k+1}}\),
and since $S_{h_{k+2}}$ is ``well contained'' inside $S_{h_{k+1}}$, by classical Pogorelov and Schauder estimates we get
\begin{eqnarray}
\|D^{2}v_{k}-D^{2}v_{k+1}\|_{C^{0}(S_{h_{k+2}})}\le C_{\bar K}' h^{\sigma}_{k} \label{C2}\\
\|D^{3}v_{k}-D^{3}v_{k+1}\|_{C^{0}(S_{h_{k+2}})}\le C_{\bar K}' h^{\sigma-1/2}_{k}\label{C3},
\end{eqnarray}
where $C_{\bar K}'$ is some constant depending only on $\bar K$.
By \eqref{C2} applied to $v_j$ for all $j=1,\ldots,k$ (this can be done since, by the inductive assumption, \eqref{goodshape} holds for $h=h_j$
with $j=1,\ldots,k$) we obtain 
\[
\begin{split}
|D^{2}v_{k+1}(0)|&\le |D^{2}v_{1}(0)|+\sum_{j=1}^{k}|D^{2} v_{j}(0)-D^{2} v_{j+1}(0)|\\
&\le M+C_{\bar K}'h^{\sigma}_{1}\sum_{j=0}^{k}2^{-j\sigma}\\
&\le M+\frac{C'_{\bar K}}{1-2^{-\sigma}}h^{\sigma}_{1} \le M+1,
\end{split}
\]
provided we choose \(h_{1}\) small enough (recall that \(h_{k}=h_{1}2^{-k}\)). By the definition of \(\bar K\) it follows that also \(S_{h_{k+1}}\) satisfies \eqref{goodshape}, concluding the proof of the inductive step.

\medskip
\noindent
$\bullet$ \textit{Step 2: higher regularity}.
Now that we know that $u \in C^{1,1}(B_{1/8})$, Equation \eqref{eq:MA} becomes uniformly elliptic.
So one may use Evans-Krylov Theorem to obtain that $u \in C_{\rm loc}^{2,\sigma'}(B_{1/9})$ for some $\sigma'>0$, and then standard Schauder estimates
to conclude the proof.
However, for the convenience of the reader, we show here how to give a simple direct proof of the $C^{2,\sigma'}$ regularity of $u$ with $\sigma'=2\sigma$.

As in the previous step, it suffices to show that \(u\) is \(C^{2,\sigma'}\) at the origin, and for this we have to prove that
there exists a sequence of paraboloids \(P_{k}\) such that
\begin{equation}\label{pisa0}
\sup_{B_{r^{k}_{0}/C}}|u-P_{k}|\le C r_{0}^{k(2+\sigma')}
\end{equation}
for some \(r_{0},C>0\). 

Let \(v_{k}\) be as in the previous step, and let \(P_{k}\) be their second order Taylor expansion at $\zero$:
\[
P_{k}(x)=v_{k}(\zero)+\nabla v_{k}(\zero)\cdot x+ \frac 1 2 D^{2} v_{k}(\zero)x\cdot x.
\]
We observe that, thanks to \eqref{goodshape},
\begin{equation}\label{pisa1}
 \|v_{k}-P_{k}\|_{C^{0}(B(0,\sqrt{h_{k+2}}/K))}\le \|v_{k}-P_{k}\|_{C^{0}(S_{h_{k+2}})}\le  C \|D^{3}v_{k}\|_{C^{0}(S_{h_{k+2}})} h^{3/2}_{k}.
\end{equation}
In addition, by \eqref{C3} applied with $j=1,\ldots,k$ and recalling that \(h_{k}=h_{1}2^{-k}\) and \(2\sigma<1\) (see \eqref{sigma}), we get
\begin{equation}\label{pisa2}
\begin{split}
\|D^{3}v_{k}\|_{C^{0}(S_{h_{k+2}})}&\le \|D^{3}v_{1}\|_{C^{0}(S_{h_{3}})}+\sum_{j=1}^{k} \|D^{3}v_{j}-D^{3}v_{j+1}\|_{C^{0}(S_{h_{j+2}})}\\
&\le C\Big(1+\sum_{j=1}^{k} h_j^{(\sigma-1/2)}\Big)\le C h^{\sigma-1/2}_{k}.
\end{split}
\end{equation}
Combining  \eqref{goodshape}, \eqref{pisa1}, \eqref{pisa2},
and recalling \eqref{eq:close} and \eqref{sigma}, we obtain
$$
\|u-P_{k}\|_{C^{0}(B_{\sqrt{h_{k+2}}/K})}\le \|v_{k}-P_{k}\|_{C^{0}(S_{h_{k+2}})}+\|v_{k}-u\|_{C^{0}(S_{h_{k+2}})}\le Ch^{1+\sigma}_{k},
$$
so \eqref{pisa0} follows with \(r_{0}=1/\sqrt{2}\) and \(\sigma'=2\sigma\).

\end{proof}

\end{document}